\documentclass[a4paper,reqno]{amsart}

\usepackage[british]{babel}
\usepackage{amssymb,amsmath,amsbsy,mathtools}
\usepackage{hyphenat}
\usepackage{upgreek}
\usepackage{dsfont}
\usepackage{graphicx,color,nicefrac} 
\usepackage{amsmath}
\usepackage{wasysym}
\usepackage{tikz}
\usepackage{pgfplots}
\usepackage[hidelinks]{hyperref}

\definecolor{cmblue}{RGB}{64,97,190}
\definecolor{cmgreen}{RGB}{125,190,64}
\definecolor{cmred}{RGB}{190,64,64}
\definecolor{cmmagenta}{RGB}{126,64,190}

\newtheorem{theorem}{Theorem}[section]
\newtheorem{lemma}[theorem]{Lemma}

\newtheorem{proposition}[theorem]{Proposition}
\newtheorem{remark}[theorem]{Remark}

\theoremstyle{definition}
\newtheorem{definition}[theorem]{Definition}

\DeclareMathOperator{\Diam}{diam}
\DeclareMathOperator{\Span}{span}
\DeclareMathOperator{\Supp}{supp}

\newcommand{\isdef}{\mathrel{\mathrel{\mathop:}=}}
\newcommand{\dint}{\operatorname{d}\!}
\renewcommand{\subset}{\subseteq}

\renewcommand{\epsilon}{\varepsilon}

\newcommand{\nablabfm}{\boldsymbol{\nabla}}
\newcommand{\lambdabfm}{\boldsymbol{\lambda}}

\newcommand{\Cbbb}{\mathbb{C}}
\newcommand{\Nbbb}{\mathbb{N}}
\newcommand{\Rbbb}{\mathbb{R}}
\newcommand{\Zbbb}{\mathbb{Z}}

\newcommand{\Acal}{\mathcal{A}}
\newcommand{\Ocal}{\mathcal{O}}

\newcommand{\bfrak}{\mathfrak{b}}
\newcommand{\Afrak}{\mathfrak{A}}
\newcommand{\Bfrak}{\mathfrak{B}}
\newcommand{\Hfrak}{\mathfrak{H}}

\newcommand{\bbfm}{\mathbf{b}}
\newcommand{\cbfm}{\mathbf{c}}
\newcommand{\dbfm}{\mathbf{d}}
\newcommand{\ebfm}{\mathbf{e}}
\newcommand{\jbfm}{\mathbf{j}}
\newcommand{\ubfm}{\mathbf{u}}
\newcommand{\vbfm}{\mathbf{v}}
\newcommand{\Abfm}{\mathbf{A}}
\newcommand{\Bbfm}{\mathbf{B}}
\newcommand{\Ibfm}{\mathbf{I}}
\newcommand{\Mbfm}{\mathbf{M}}
\newcommand{\Tbfm}{\mathbf{T}}
\newcommand{\zerobfm}{\mathbf{0}}
\newcommand{\onebfm}{\mathbf{1}}

\newcommand{\xvec}{\boldsymbol{x}}

\title[On hybrid Sobolev and Besov spaces]{On Sobolev and 
Besov spaces of hybrid regularity}
\author{Helmut Harbrecht}
\author{Remo von Rickenbach}
\address{Helmut Harbrecht and Remo von Rickenbach,
	Departement Mathematik und Informatik,
	Universit\"at Basel,
Spiegelgasse 1, 4051 Basel}
\email{\{helmut.harbrecht,remo.vonrickenbach\}@unibas.ch}

\thanks{This research has been supported by the Swiss National Science 
	Foundation (SNSF) through the project ``Adaptive Boundary 
	Element Methods Using Anisotropic Wavelets''  under grant 
agreement no.\ 200021\_192041.}

\begin{document}

\begin{abstract}
	The present article is concerned with the nonlinear approximation of
	functions in the Sobolev space \(H^{q}\) with respect to a tensor-product, 
	or hyperbolic wavelet basis on the unit \(n\)-cube. Here, \(q\) is a real
	number, which is not necessarily positive. We derive Jackson and 
	Bernstein inequalities to obtain that the approximation classes 
	contain Besov spaces of hybrid regularity. 
	Especially, we show that all functions that can be approximated by classical
	wavelets are also approximable by tensor-product wavelets at least at
	the same rate.
	In particular, this implies that for nonnegative regularity, the classical
	Besov spaces of regularity \(B^{q+sn, \tau}_{\tau}\), with
	\(\frac{1}{\tau} = s +\frac{1}{2}\), are included in the Besov spaces of
	hybrid regularity \(\Bfrak^{q,s,\tau}_{\tau}\),
	with isotropic regularity \(q\) and additional mixed regularity \(s\).
\end{abstract}

\maketitle

\section{Introduction}\label{sec:Introduction}

If we want to approximate a function, there are many different methods. 
The best known and understood method is linear approximation. In this 
setting, given a function \(u \in V\), we take a sequence of nested, linear 
trial spaces \(V_j\subset V\) and intend to quantify the best approximation 
error \(\inf_{v_j \in V_j} \|u-v_j\|_V\) with respect to \(j\). In order to guarantee 
a certain convergence order, specific constraints on the target function 
\(u\) have to be satisfied. For example, to approximate a function 
\(u\) in the (isotropic) Sobolev space \(V = H^q(\square)\), with \(\square
\isdef (0, 1)^n\),
by piecewise polynomial functions of order \(d\) defined on a 
quasi-uniform mesh with support length \(2^{-j}\), we require that 
\(u \in H^s(\square)\) to expect the convergence order 
\(2^{-(s-q)j}\) for \(q \leq s \leq d\). When comparing the number 
of required trial functions with the accuracy, any function in 
\(H^{q+ns}(\square)\) is asymptotically approximable 
by \(N\) terms at the rate \(N^{-s}\). 

In the case of sufficiently high regularity, this approach works perfectly fine.
On the other hand, if only limited regularity of the function \(u\) is provided,
the optimal convergence rate cannot be realised.
Therefore, the framework of \emph{nonlinear approximation} was developed to
approximate a function adaptively, see e.g.\ \cite{DeV98} for an overview.
In this setting, the trial spaces used for approximating \(u\) are no longer
linear subspaces.
In particular, for \emph{best \(N\)-term approximation}, the trial space
\(V_N\) is the space of linear combinations from a dictionary,
consisting of at most \(N\) terms.
It is easy to see that the best nonlinear approximant consisting of \(N\) terms
is at least as good as every best linear approximant consisting of \(N\) terms, as
we can always restrict \(V_N\) to be a linear space.

A basic tool for nonlinear approximation is a Riesz basis 
\(\Psi = \{\psi_{\lambda} : \lambda \in\nabla\}\) for the space 
\(V\), which can, if \(V\) is Sobolev space, be realised by a wavelet 
basis, cf.\ \cite{Dah97,Dau92,DeV98,Sch98} for example.
With a Riesz basis at hand, approximating the function \(u\) in \(V\) 
by \(N\) terms from \(\Psi\) is equivalent to approximating the 
coefficient vector \(\ubfm\) in \(\ell^2(\nabla)\). 
By using a wavelet basis one can show, cf.\ \cite{DeV98} and the
references therein, that the requirements to achieve the rate
\(N^{-s}\) are much weaker compared to linear approximation: 
In contrast to Sobolev regularity of order \(q +sn\), only 
Besov regularity of order \(q +sn\) and integrability \(\tau\isdef 
(\frac{1}{2}+s)^{-1}\) is required. Since then \(\tau \leq 2\),
we can conclude that these spaces contain the respective
Sobolev space \(H^{q+sn}(\square)\), but are in general
much larger.

This classical result holds if the space \(H^{q}(\square)\) is discretised
by isotropic wavelets. The recent articles \cite{BHW23,HvR24,SUV21} 
show, however, that the classical Sobolev spaces \(H^{q}(\square)\) 
can also be characterised by tensor-product or hyperbolic wavelets.
This approach leads to a new perspective on the best \(N\)-term approximation,
as it is well-known that tensor-product wavelets can approximate functions
essentially better. This concept is known as \emph{sparse grid}, cf.\
\cite{BG04,GH13}. Here, the term \emph{essentially} is understood 
as \(n\)-times as well up to logarithmic terms.

However, this is only true under certain additional requirements on the target
function. For example, if \(q = 0\), meaning that \(H^{q}(\square) =
L^{2}(\square)\), the target function needs to admit dominating mixed Sobolev
smoothness. As a consequence, for the best \(N\)-term approximation,
one needs to consider Besov spaces with dominating mixed regularity, which have,
for instance, been studied in \cite{Han10,NS17,SU09,Vyb05}.

In \cite{Nit04,Nit06}, Besov spaces have been used for the approximation with
tensor-product wavelets in \(L^2(\square)\) and \(H^{1}(\square)\).
Therein, the respective approximation spaces have been defined 
as the tensor product of quasi-Banach spaces, resulting in 
function spaces of Besov type of hybrid regularity.
If \(q \geq 0\), this procedure can easily be extended to 
\(H^{q}(\square)\), but whenever \(q\) is strictly positive, 
the approximation spaces are of hybrid regularity.
However, if one wishes to approximate a function in 
\(H^{q}(\square)\) for negative \(q\), which is necessary in the case 
of e.g.\ boundary integral equations, cf.\ \cite{SS11,Ste08}, these 
results do not carry over, as the resulting approximation spaces 
can no longer be written as the intersection of tensor-product spaces.

In contrast, the hybrid regularity Besov spaces on the whole space
\(\Rbbb^n\) have been introduced in terms of wavelet coefficients in 
\cite{BHW23}. Therein, upper and lower bounds on the Kolmogorov 
dictionary width and the best \(N\)-term approximation have 
been derived. However, those results require a 
difference in the isotropic part of the hybrid regularity. As we 
will see, this is also not the case when we consider the best 
\(N\)-term approximation in a Hilbert space \(H^{q}(\square)\) 
with respect to a tensor-product wavelet basis.

In this article, we will characterise the approximation spaces 
\(\Acal^{s}\big(H^{q}(\square)\big)\) consisting of functions \(u \in
H^{q}(\square)\), which can be approximated by \(N\) terms at the 
rate \(N^{-s}\). As we will see, the resulting spaces contain the 
Besov spaces of hybrid regularity from \cite{BHW23}. 
Additionally, when requiring slightly more regularity in terms of
logarithmic decay of the coefficients, with the help of \cite{ACJRV15}
we can immediately conclude that these spaces can be nested 
between classical Besov spaces. However, to 
the authors' best knowledge, it was not known yet whether these
hybrid regularity Besov spaces are embedded in the corresponding 
classical Besov spaces, which turns out to be true in the setting 
of best \(N\)-term approximation. Vice versa, also the opposite 
natural embedding of a classical Besov space of regularity \(q+s\) 
into the hybrid regularity Besov space of isotropic regularity \(q\) 
and additional mixed regularity \(s\) will be proven.

The rest of this article is organised as follows: We introduce the 
multiscale hierarchy and state the requirements on the wavelets
under consideration
in Section \ref{sec:Problem_Forumlation}. In Section 
\ref{sec:Function_Spaces}, we define the function spaces used 
for the approximation with isotropic and tensor-product wavelets. 
This topic, together with a brief review of interpolation, is treated 
in Section \ref{sec:Approximation_Interpolation}. Afterwards, we 
compare the isotropic and hyperbolic approximation spaces in 
Section \ref{sec:comparison_anisotropic_isotropic}, and we state concluding 
remarks in Section \ref{sec:conclusion}.

Throughout this article, to avoid the repeated use of unspecified 
generic constants, we write \(A \lesssim B\) if \(A\) is bounded 
by a uniform constant times \(B\), where the constant does not 
depend on any parametres which \(A\) and \(B\) might depend 
on. Similarly, we write \(A \gtrsim B\) if and only if \(B \lesssim A\).
Finally, if \(A \lesssim B\) and \(B \lesssim A\), we write \(A \sim B\).

\section{Wavelet Bases}
\label{sec:Problem_Forumlation}

In this section, we define the wavelet bases under consideration
and state their most important properties. Throughout the article, 
we assume that the scaling functions and the wavelets involved 
are compactly supported or decay sufficiently fast.
Moreover, we require that this property also holds 
for the dual scaling functions and the dual wavelets.
Numerically feasible wavelets are, in general, biorthogonal 
and were first constructed in \cite{CDF92}. 
These wavelets yield a compact support and therefore,
\(\Diam (\Supp \phi_{\lambda}) 
\sim \Diam (\Supp \psi_{\lambda}) \sim 2^{-|\lambda|}\) holds
for any one-dimensional scaling function \(\phi_{\lambda}\) 
and wavelet \(\psi_{\lambda}\) on level \(|\lambda|\). 
Later, this construction was also transferred 
to a finite interval in \cite{DKU99}. 
Alternatively, as shown in \cite{CDV93,Mey91}, it is also possible 
to use Daubechies wavelets \cite{Dau88,Dau92} on the interval. 
However, as for numerical applications there is a need for an efficient evaluation 
of the primal wavelets, we especially emphasise the constructions 
of \cite{CDF92,DKU99}.

\subsection{Univariate Wavelet Bases}
\label{sec:Univariate_Wavelet_Bases}

We consider a sequence of nested, finite\hyp{}dimensional, and asymptotically 
dense function spaces
\begin{equation*}
	V_{j_0} \subset V_{j_0+1} \subset \ldots \subset V_{j-1} \subset V_j 
	\subset V_{j+1} \ldots \subset V,
\end{equation*}
which are used to discretise a vector space \(V\)
consisting of functions (or distributions) on the unit interval \(I \isdef [0, 1]\).
Typically, \(V = H^{q}([0,1])\) is a Sobolev space with regularity \(q\).
We assume that the function spaces \(V_j\) can be generated by
shifts and dyadic dilations of scaling functions, with 
possible modifications at the endpoints of the interval. 
Moreover, 
for a suitable index set \(\Delta_j\), we assume that
\begin{equation*}
	\Phi_j \isdef \big\{ \phi_{\lambda} : \lambda \in \Delta_j \big\}
\end{equation*}
is a Riesz basis of \(V_j\).
Here, the index \(\lambda = (j,k)\) contains information 
of the level \(j\) and the location \(k\). 
Let us mention that the spaces \((V_j)_{j \geq j_0}\) form a
\emph{multiresolution analysis} on the unit interval $[0,1]$, see e.g.~\cite{Mal89}.
In the easiest case, one can think of \(\phi\) 
as the constant function \(1\), and of \(\phi_{\lambda}\) as a 
properly scaled, dyadic indicator function, i.e., 
\(\phi_{\lambda} = 2^{\nicefrac{j}{2}} \mathds{1}_{[2^{-j}k, 
\,2^{-j}(k+1)]} = 2^{\nicefrac{j}{2}} \mathds{1}_{I}(2^{j}x - k)\).

We say that the spaces \(V_j\) have the approximation order \(d\) 
if they contain locally all polynomials up to the order \(d\). Moreover, 
\(V_j\) are said to have the regularity \(\gamma 
\isdef \sup \{ s \in \Rbbb : V_j \subset H^s([0, 1])\}\).

\subsubsection{Multiscale Bases on \([0,1]\)}

As the spaces \(V_j\) are nested, we may write
\begin{equation}
	V_j = V_{j-1} \oplus W_j
	\label{eq:complement_decomposition}
\end{equation}
with the complement or difference space \(W_j\).
One can show that, if the scaling function \(\phi\) 
generates a shift-invariant space, cf.\ \cite{CDF92,DKU99},
there exist wavelet functions such that
\begin{equation*}
	\Psi_j \isdef \big\{ \psi_{\lambda} : \lambda \in \nabla_j \big\}
\end{equation*}
is a basis set of \(W_j\). Also herein, \(\nabla_j\) is a suitable 
index set and \(\psi_{\lambda}\) is a properly scaled and 
translated copy of a mother wavelet (again with possible
modifications at the endpoints of the interval). For convenience, 
if \(\lambda \in \Delta_j\) or \(\lambda\in\nabla_j\), let us denote 
its level by \(|\lambda| \isdef j\). 

From \eqref{eq:complement_decomposition}, we recursively 
obtain the multiscale decomposition
\begin{equation*}
	V_j = V_{j_0} \oplus W_{j_0+1} \oplus \ldots \oplus W_j,
\end{equation*}
provided that \(j_0 < j\). 
If we also define \(W_{j_0} \isdef V_{j_0}\),
\(\nabla_{j_0} \isdef \Delta_{j_0}\), and \(\psi_{\lambda} \isdef
\phi_{\lambda}\) for \(\lambda \in \nabla_{j_0}\),
as the function spaces \(V_j\) 
are asymptotically dense, we conclude that the set
\begin{equation}
	\Psi \isdef 
	\{\psi_{\lambda} : \lambda \in \nabla\},
	\qquad \nabla \isdef \bigcup_{j = j_0}^{\infty} \nabla_j,
	\label{eq:multiscale_basis}
\end{equation}
spans a dense subset of \(V\).

Following \cite{CDV93,DKU99,Mey91} for example,
the set
\(\Psi\) forms a Riesz basis of \(L^{2}([0,1])\), meaning that
\begin{equation*}
	\left\| \sum_{\lambda \in \nabla} c_{\lambda} \psi_{\lambda}
	\right\|_{L^{2}([0,1])}^2
	\sim \sum_{\lambda \in \nabla} \big|c_{\lambda}\big|^2.
\end{equation*}
Hence, there exists a biorthogonal multiresolution analysis
\begin{equation*}
	\tilde{V}_{j_0} \subset\tilde{V}_{j_0+1} \subset \ldots \subset\tilde{V}_{j-1} \subset\tilde{V}_j 
	\subset\tilde{V}_{j+1} \ldots \subset V'
\end{equation*}
which is also a Riesz basis for \(L^{2}([0,1])\) and asymptotically
dense in \(V'\).

The spaces \(\tilde{V}_j\isdef\{\tilde\phi_{\lambda}: 
\lambda\in\Delta_j\}\) admit the regularity \(\tilde{\gamma} > 0\) 
and the approximation order \(\tilde{d}\). This fact provides us 
the number of vanishing moments of the wavelets \(\psi_{\lambda}\), 
which means that \(\langle p, \, \psi_{\lambda} \rangle = 0\) 
for any polynomial \(p\) up to the order \(\tilde{d}\). 
Finally, there exists also a unique 
biorthogonal wavelet basis \(\tilde{\Psi} = \{\tilde{\psi}_{\lambda} : 
\lambda\in\nabla\}\), satisfying
\begin{equation}
	\big\langle \tilde{\psi}_{\lambda'}, \psi_{\lambda} \big\rangle =
	\delta_{\lambda,\lambda'}.
	\label{eq:biorthogonality}
\end{equation}

\begin{figure}[hbt]
	\begin{minipage}{0.45\textwidth}
		\centering
		\scalebox{1.0}{\noindent
\begin{tikzpicture}

	\tikzstyle{every node}=[font=\small]

	\begin{axis}[
			xmin=0, xmax=1,
			ymin=-4, ymax = 4,
			width=\textwidth,
			height=1.4\textwidth,
			xtick = {0,0.25,0.5,0.75,1},
		]

		\addplot[
			color=cmblue,
		]
		coordinates{
			(0,1.76776695296637)
			(0.125,1.76776695296637)
			(0.125,-3.88908729652601)
			(0.25,-3.88908729652601)
			(0.25,1.41421356237309)
			(0.5,1.41421356237309)
			(0.5,-0.353553390593274)
			(0.75,-0.353553390593274)
			(0.75,0)
		};

		\addplot[
			color=cmred,
		]
		coordinates {
			(0,-0.353553390593274)
			(0.25,-0.353553390593274)
			(0.25,2.82842712474619)
			(0.375,2.82842712474619)
			(0.375,-2.82842712474619)
			(0.5,-2.82842712474619)
			(0.5,0.353553390593274)
			(0.75,0.353553390593274)
			(0.75,0)
		};

		\addplot[
			color=cmgreen,
		]
		coordinates {
			(0.25,0)
			(0.25,-0.353553390593274)
			(0.5,-0.353553390593274)
			(0.5,2.82842712474619)
			(0.625,2.82842712474619)
			(0.625,-2.82842712474619)
			(0.75,-2.82842712474619)
			(0.75,0.353553390593274)
			(1,0.353553390593274)
		};

		\addplot[
			color=cmmagenta,
		]
		coordinates {
			(0.25,0)
			(0.25,-0.353553390593274)
			(0.5,-0.353553390593274)
			(0.5,1.41421356237309)
			(0.75,1.41421356237309)
			(0.75,-3.88908729652601)
			(0.875,-3.88908729652601)
			(0.875,1.76776695296637)
			(1,1.76776695296637)
		};

	\end{axis}
\end{tikzpicture}}
	\end{minipage}
	\hfill
	\begin{minipage}{0.45\textwidth}
		\centering
		\scalebox{1.0}{\input{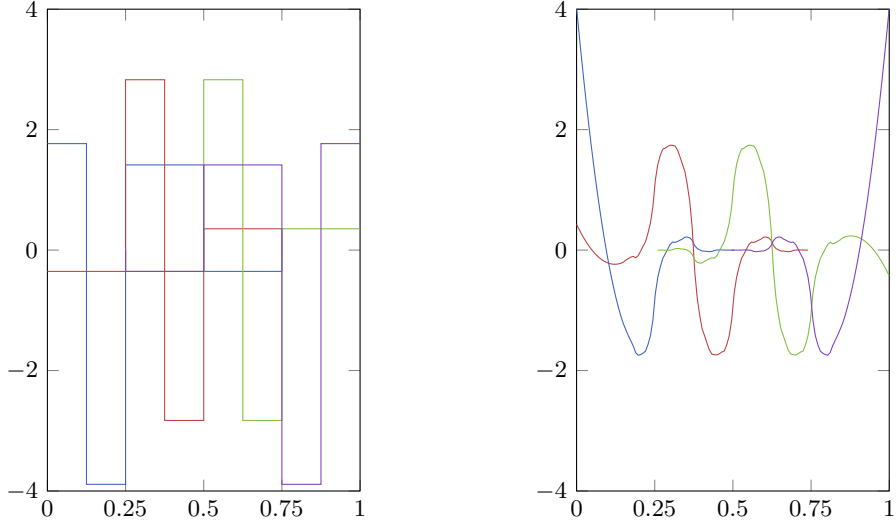}}
	\end{minipage}
	\caption{Primal wavelets (left) and dual wavelets (right) according to the
		construction in \cite{DKU99}. In this setting, we have \(d = 1\), \(\gamma =
	\frac{1}{2}\), and \(\tilde{d} = 3\).}
	\label{fig:wavelet_13}
\end{figure}
\begin{figure}[hbt]
	\begin{minipage}{0.45\textwidth}
		\centering
		\scalebox{1.0}{\noindent
\begin{tikzpicture}

	\tikzstyle{every node}=[font=\small]

	\begin{axis}[
			xmin=0, xmax=1,
			ymin=-4, ymax = 4,
			width=\textwidth,
			height=1.4\textwidth,
			xtick = {0,0.25,0.5,0.75,1},
		]

		\addplot[
			color=cmblue,
		]
		coordinates{
			(0,0)
			(0.125,1.76776695296637)
			(0.25,-2.12132034355964)
			(0.375,-0.707106781186548)
			(0.5,0.707106781186548)
			(0.625,0.353553390593274)
			(0.75,0)
		};

		\addplot[
			color=cmred,
		]
		coordinates {
			(0,0)
			(0.125,-0.353553390593274)
			(0.25,-0.707106781186548)
			(0.375,2.12132034355964)
			(0.5,-0.707106781186548)
			(0.625,-0.353553390593274)
			(0.75,0)
		};

		\addplot[
			color=cmgreen,
		]
		coordinates {
			(0.25, 0)
			(0.375,-0.353553390593274)
			(0.5,-0.707106781186548)
			(0.625,2.12132034355964)
			(0.75,-0.707106781186548)
			(0.875,-0.353553390593274)
			(1,0)
		};

		\addplot[
			color=cmmagenta,
		]
		coordinates {
			(0.5,0)
			(0.625,-0.176776695296637)
			(0.75,-0.353553390593274)
			(0.875,1.59099025766973)
			(1,-2.12132034355964)
		};
	\end{axis}
\end{tikzpicture}}
	\end{minipage}
	\hfill
	\begin{minipage}{0.45\textwidth}
		\centering
		\scalebox{1.0}{\input{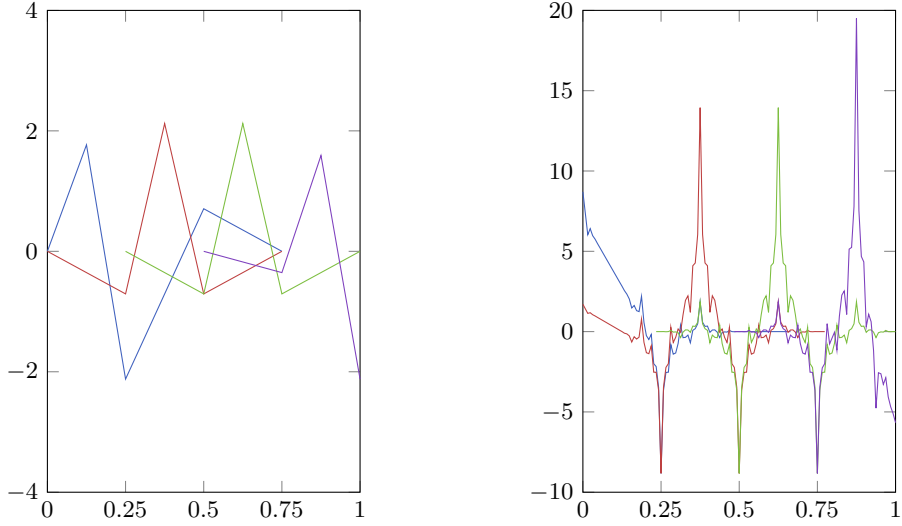}}
	\end{minipage}
	\caption{Primal wavelets (left) and dual wavelets (right) with complementary
		boundary conditions according to the construction in \cite{DS98}. 
		In this setting, we have \(d = \tilde{d} = 2\), and \(\gamma = \frac{3}{2}\). 
		The primal wavelets satisfy zero boundary conditions at \(x=0\) 
	and the dual ansatz functions at \(x=1\).}
	\label{fig:wavelet_22}
\end{figure}

For the setting of \cite{DKU99,DS98}, i.e.,
compactly supported, piecewise constant wavelets with three 
vanishing moments \((d,\tilde{d}) = (1,3)\) and piecewise 
linear wavelets with two vanishing moments \((d,\tilde{d}) = (2,2)\), 
an illustration is provided in Figure \ref{fig:wavelet_13} and Figure
\ref{fig:wavelet_22}, respectively. Notice that in the latter figure, 
the primal wavelets depicted discretise functions with zero 
boundary conditions at \(x = 0\) and nonzero boundary conditions 
at \(x = 1\).

\subsubsection{Multiscale Transforms}

By considering the basis sets \(\Phi_j, \Psi_j\) as row vectors,
we may, due to the relation \eqref{eq:complement_decomposition}, 
write
\begin{align}
	\label{eq:Multiscale_primal}
	\Phi_{j-1} = \Phi_j \Mbfm_{j,0}, \quad \Psi_j = \Phi_j \Mbfm_{j,1},\\
	\label{eq:Mutltscale_dual}
	\tilde{\Phi}_{j-1} = \tilde{\Phi}_j \tilde{\Mbfm}_{j,0}, 
	\quad \tilde{\Psi}_j = \tilde{\Phi}_j \tilde{\Mbfm}_{j,1}.
\end{align}
Herein, the matrices
\(\Mbfm_{j,0}, \tilde{\Mbfm}_{j,0} \in
\Rbbb^{|\Delta_j| \times |\Delta_{j-1}|}\) and
\(\Mbfm_{j,1}, \tilde{\Mbfm}_{j,1} \in
\Rbbb^{|\Delta_j| \times |\nabla_j|}\)
are called the \emph{refinement masks}. By the local supports
and the chosen scaling, the matrices \(\big[\Mbfm_{j,0}, \Mbfm_{j,1}\big]\) 
and \(\big[\tilde{\Mbfm}_{j,0}, \tilde{\Mbfm}_{j,1}\big]\) are
uniformly stable in the sense that their condition numbers are
uniformly bounded.

If \(u \in V_j\), we can write
\begin{equation}
	\label{eq:u_sum_coefficients}
	u = \sum_{\lambda \in \Delta_j} c_{\lambda} \phi_{\lambda}
	= \sum_{\lambda \in \Delta_{j-1}} c_{\lambda} \phi_{\lambda} +
	\sum_{\lambda \in \nabla_j} d_{\lambda} \psi_{\lambda}.
\end{equation}
By interpreting the coefficients \(c_{\lambda}\) and
\(d_{\lambda}\) as finite-dimensional column vectors, 
we may use \eqref{eq:Multiscale_primal} to deduce that
\eqref{eq:u_sum_coefficients} can equivalently be written as
\begin{equation*}
	u = \Phi_j \cbfm_j = \big[\Phi_{j-1}, \Psi_{j}\big]
	\begin{bmatrix}
		\cbfm_{j-1}\\
		\dbfm_{j}\\
	\end{bmatrix}
	= \Phi_j \big[\Mbfm_{j,0}, \Mbfm_{j,1}\big]
	\begin{bmatrix}
		\cbfm_{j-1}\\
		\dbfm_{j}\\
	\end{bmatrix},
\end{equation*}
where \(\cbfm_j \isdef [c_{\lambda}]_{\lambda \in \Delta_j}\)
and \(\dbfm_j \isdef [d_{\lambda}]_{\lambda \in \nabla_j}\).
On the other hand, remarking that 
\(\cbfm_j = \langle \tilde{\Phi}_j, u\rangle\) and
\(\dbfm_j = \langle \tilde{\Psi}_j, u\rangle\), we obtain that
\begin{align*}
	\begin{bmatrix}
		\cbfm_{j-1}\\
		\dbfm_{j}\\
	\end{bmatrix} = 
	\begin{bmatrix}
		\tilde{\Mbfm}_{j,0}^{\intercal}\\
		\tilde{\Mbfm}_{j,1}^{\intercal}\\
	\end{bmatrix} \cbfm_j.
\end{align*}
Hence, we conclude that \([\tilde{\Mbfm}_{j,0},
	\tilde{\Mbfm}_{j,1}]^{\intercal} [\Mbfm_{j,0},
\Mbfm_{j,1}] = \Ibfm\), and therefore
\begin{equation*}
	\begin{aligned}
		\tilde{\Mbfm}_{j,0}^{\intercal} \Mbfm_{j,0} &= \Ibfm, \quad
		&\tilde{\Mbfm}_{j,1}^{\intercal} \Mbfm_{j,1} &= \Ibfm, \\
		\tilde{\Mbfm}_{j,0}^{\intercal} \Mbfm_{j,1} &= \zerobfm, \quad
		&\tilde{\Mbfm}_{j,1}^{\intercal} \Mbfm_{j,0} &= \zerobfm.
	\end{aligned}
\end{equation*}

By applying the multiscale transform recursively, there follows
\begin{align}
	\big[ \Phi_{j_0}, \Psi_{j_0+1}, \dots, \Psi_{j} \big]
	= \Phi_j \Tbfm_j 
	\isdef \Phi_j \prod_{\ell = j_0+1}^{j}
	\begin{bmatrix}
		[\Mbfm_{\ell,0}, \Mbfm_{\ell, 1}] & \\
		& \Ibfm_{|\Delta_j| - |\Delta_{\ell}|}
	\end{bmatrix},
	\label{eq:def_of_T}
\end{align}
and accordingly on the dual side
\begin{equation}
	\begin{bmatrix}
		\cbfm_{j_0} \\
		\dbfm_{j_0+1} \\
		\vdots \\
		\dbfm_{j} \\
	\end{bmatrix}  
	= \tilde{\Tbfm}_j^{\intercal} \cbfm_j
	\isdef \prod_{\ell = j_0+1}^{j}
	\begin{bmatrix}
		\left[
			\begin{smallmatrix}
				\tilde{\Mbfm}_{\ell, 0}^{\intercal}\\ 
				\tilde{\Mbfm}_{\ell, 1}^{\intercal}\\
		\end{smallmatrix}\right] & \\
		& \Ibfm_{|\Delta_j| - |\Delta_{\ell}|}\\
	\end{bmatrix} \cbfm_j.
	\label{eq:def_of_tilde_T}
\end{equation}

\subsection{Multivariate Wavelet Bases}
\label{sec:multivariate_wavelet_bases}

In this subsection, we will briefly review the construction of
multivariate wavelet bases on the unit cube \(\square = [0,1]^n\). 
There are two well-established ways to create wavelet
bases on \(\square\). On the one hand, one can construct
\emph{isotropic} wavelet bases, i.e., bases consisting of functions 
whose supports are shape-regular cuboids. On the other hand, one 
can use the tensor product of wavelets on different levels, resulting 
in so-called \emph{anisotropic} or \emph{hyperbolic} wavelet bases.
In both settings, we also start with a multiresolution analysis where 
on each level \(j\), the \(n\)-fold tensor product of the spaces \(V_j\) 
defined in Section \ref{sec:Univariate_Wavelet_Bases} is involved. 

\subsubsection{Isotropic Wavelet Bases}

Let us first address the bivariate case. 
It is trivial to deduce that 
the basis set \(\Phi_j\otimes \Phi_j\) 
discretises the space \(V_j \otimes V_j\)
for any \(j \geq j_0\). 
As we also have
\begin{equation*}
	V_j \otimes V_j = \big( V_{j-1} \otimes V_{j-1}\big) \oplus 
	\big( V_{j-1} \otimes W_{j}\big) \oplus
	\big( W_{j} \otimes V_{j-1}\big) \oplus 
	\big( W_{j} \otimes W_{j}\big)
\end{equation*}
for \(j \geq j_0+1\), 
we can define
\begin{equation*}
	\Theta_j \isdef \Theta_j^{(0,1)} 
	\cup \Theta_j^{(1,0)} \cup \Theta_j^{(1,1)} \isdef
	\big(\Phi_{j-1} \otimes \Psi_j\big) \cup 
	\big(\Psi_j \otimes \Phi_{j-1}\big) \cup
	\big(\Psi_j \otimes \Psi_j\big)
\end{equation*}
as a basis of the complement space \((V_j \otimes V_j) 
\ominus (V_{j-1} \otimes V_{j-1})\).

This procedure can be generalised to the \(n\)-variate case as well by defining
\begin{align*}
	\Theta_{j}^{\ebfm} \isdef \Theta_{j}^{e_1} \otimes \dots 
	\otimes \Theta_{j}^{e_n}, \qquad
	\Theta_{j}^{e} = 
	\begin{cases}
		\Phi_{j-1}, &e = 0,\\
		\Psi_j, &e = 1,
	\end{cases}
\end{align*}
by which the basis set
\begin{equation*}
	\Theta_j \isdef \bigcup_{\ebfm \in \{0,1\}^n \setminus \{\zerobfm\}}
	\Theta_j^{\ebfm}
\end{equation*}
spans the complement space.
If we define \(\Theta_{j_0} \isdef \Phi_{j_0} \otimes \ldots \otimes
\Phi_{j_0}\), then there holds 
\begin{equation*}
	\Span\left\{\bigcup_{j = j_0}^m  \Theta_j\right\} = 
	V_m \otimes \dots \otimes V_m,
\end{equation*}
and these sets form a Riesz basis of \(L^2(\square)\)
for \(m \to \infty\).

For convenience, let us also define the index set
\(\hexagon_j^{\ebfm}\) as the set of
indices of wavelets \(\theta_{\mu} \in \Theta_j^{\ebfm}\), as well as
\begin{equation*}
	\hexagon_j \isdef \bigcup_{\ebfm \in\{0,1\}^n \setminus
	\{\zerobfm\}} \hexagon_j^{\ebfm}, \qquad
	\hexagon \isdef \bigcup_{j \geq j_0} \hexagon_j,
\end{equation*}
with the canonical adaptation for \(\hexagon_{j_0}\).
Finally, similar to the univariate case, 
we define the level of an index as 
\(|\mu| \isdef j\) for any \(\mu \in \hexagon_j\).

\subsubsection{Tensor-Product Wavelet Bases}

Another approach is to use the tensor product of one-dimensional 
wavelets on all the different levels. This approach is straightforward:
for a given multiindex \(\jbfm\) which satisfies the 
component-wise inequality \(\jbfm \geq \jbfm_0+\onebfm\), we define the index set 
\(\nablabfm_{\jbfm} \isdef \nabla_{j_1} \times \dots \times \nabla_{j_n}\), and for 
\(\lambdabfm = (\lambda_1, \dots, \lambda_n) \in \nablabfm_{\jbfm}\), 
we define the wavelet function
\begin{align*}
	\psi_{\lambdabfm}(\xvec) 
	\isdef \big(\psi_{\lambda_1}\otimes \dots \otimes \psi_{\lambda_n} \big) (\xvec)
	= \psi_{\lambda_1}(x_1) \dots \psi_{\lambda_n}(x_n).
\end{align*}
All such wavelets on a level \(\jbfm\) span the corresponding 
complement space, i.e.,
\[
	\Psi_{\jbfm} \isdef \big\{ \psi_{\lambdabfm} : 
	\lambdabfm \in\nablabfm_{\jbfm} \big\}
\]
is a basis set of \(W_{j_1} \otimes \dots \otimes W_{j_n}\). 
If we also re-define 
\(\psi_{\lambda} \isdef \phi_{\lambda}\) for \(\lambda \in\nabla_{j_0}\) as it
was done in the univariate case,
and extend the above definition to any multiindex \(\jbfm \geq \jbfm_0\), 
then we deduce out of \eqref{eq:multiscale_basis} that the span of
\[
	\Psi \isdef \{ \psi_{\lambdabfm} : \lambdabfm \in\nablabfm \} 
	= \bigcup_{\jbfm \geq \jbfm_0} \Psi_{\jbfm},
	\quad \nablabfm \isdef\bigcup_{\jbfm \geq \jbfm_0} \nablabfm_{\jbfm},
\]
is dense in \(V \otimes \dots \otimes V\) and forms a Riesz basis. 

\subsection{Auxiliary Results}

In this section, we state and prove four lemmata. When comparing the
different function spaces in Section \ref{sec:comparison_spaces}, they
will be crucial.

\begin{lemma}
	Let \(0 < p \leq 1\) and \(\Abfm \in \Rbbb^{m \times n}\). Then, there holds
	\begin{align}
		\|\Abfm\|_p^{p} \leq \big\|\Abfm^{\odot p}\big\|_1
		= \max_{1 \leq j \leq n} \sum_{i=1}^{m}
		\big|a_{i,j}\big|^p,
		\label{eq:matrix_p_norm}
	\end{align}
	where \(\Abfm^{\odot p}\) is the component-wise \(p\)-th power of
	\(\Abfm\).
	\label{lm:matrix_p_norm}
\end{lemma}
\begin{proof}
	Let \(\ubfm \in \Rbbb^{n}\). In view of the 
	subadditivity 
	\begin{equation}\label{eq:subadditive}
		|x+y|^p \leq |x|^p + |y|^p, \quad 0 < p \leq 1,
	\end{equation}
	we conclude
	\begin{align*}
		\|\Abfm \ubfm\|_p^{p} &= \sum_{i=1}^m \big| [\Abfm \ubfm]_i\big|^p
		=\sum_{i=1}^{m} \left| \sum_{j = 1}^{n} a_{i,j} u_j \right|^p\\
		&\leq \sum_{i=1}^{m} \sum_{j=1}^{n} \big|a_{i,j}\big|^p
		\big|u_j\big|^p\\
		&\leq \|\ubfm\|_p^{p} \max_{1 \leq j \leq n} \sum_{i=1}^{m} \big|
		a_{i,j}\big|^p.
	\end{align*}
	This implies \eqref{eq:matrix_p_norm}.
\end{proof}

\begin{remark}
	The subadditivity \eqref{eq:subadditive} implies that 
	the function \(\|\cdot\|_{p}^{p}\) is subadditive for
	\(0 < p \leq 1\), cf.\ \cite{DL93}.
	As the function \(x \mapsto x^{\nicefrac{1}{p}}\) is convex,
	the quasi-triangle inequality 
	\begin{align*}
		\|\ubfm + \vbfm\|_p \leq \big(\|\ubfm\|_{p}^{p} +
		\|\vbfm\|_{p}^{p}\big)^{\frac{1}{p}}
		= 2^{\frac{1}{p}} \bigg(\frac{1}{2} \|\ubfm\|_{p}^{p} +
		\frac{1}{2}\|\vbfm\|_{p}^{p}\bigg)^{\frac{1}{p}}
		\leq 2^{\frac{1}{p}-1} \big(\|\ubfm\|_{p} +
		\|\vbfm\|_{p}\big)
	\end{align*}
	follows for all \(\ubfm, \vbfm \in \ell^{p}\).
\end{remark}

To derive estimates on the wavelet transform matrices, 
we need to bound the matrix coefficients as well.
The decay behaviour of the coefficients is governed by the decay behaviour of
the wavelet itself.
\begin{lemma}
	\label{lm:coefficients_fafield}
	Suppose that the scaling functions and the wavelets
	satisfy the decay estimate
	\begin{equation}
		\label{eq:decay_behaviour_scaling}
		|\phi_{\lambda}(x)|, |\psi_{\lambda}(x)| 
		\lesssim 2^{\nicefrac{j}{2}} (1 + |2^{j}x - k|)^{-\alpha}, \qquad
		\lambda = (j, k), 
	\end{equation}
	and likewise for the dual side for some \(\alpha > 1\).
	Then, for \(j_0 \leq j \leq m\) and any two indices
	\(\lambda = (j,k) \in \nabla_{j}\) and \(\mu = (m, \ell) \in \Delta_{m}\),
	there holds
	\begin{equation}
		\label{eq:decay_behaviour_coefficient}
		|t_{\mu,\lambda}|, \big|\tilde{t}_{\mu,\lambda}\big|
		\lesssim 2^{\frac{j-m}{2}}
		\big(1 + |k - 2^{j-m}\ell|\big)^{-\alpha}.
	\end{equation}
\end{lemma}
\begin{proof}
	First, we note that
	\begin{align*}
		t_{\mu,\lambda} &= \int_{I} \psi_{\lambda}(x)
		\tilde{\phi}_{\mu}(x) \dint x.
	\end{align*}
	Since \(\|\psi_{\lambda}\|_{L^{\infty}(I)} \lesssim 2^{\nicefrac{j}{2}}\)
	and \(\big\|\tilde{\phi}_{\mu}\big\|_{L^{1}(I)} \lesssim 2^{-\nicefrac{m}{2}}\)
	by \eqref{eq:decay_behaviour_scaling},
	this implies \eqref{eq:decay_behaviour_coefficient} for all pairs of
	\(k\) and \(\ell\) such that \(|k - 2^{j-m}\ell| \leq 1\).

	If \(|k - 2^{j-m}\ell| \geq 1\),
	we distinguish between the following two cases.
	First, if \(|2^{m}x - \ell| \geq 2^{\frac{m-j}{\alpha}
		- 1}
	|k - 2^{j-m}\ell|\),
	there follows
	\begin{align*}
		|t_{\mu,\lambda}| &\lesssim 2^{\frac{j+m}{2}} 
		\int_{I} \big|\psi_{\lambda}(x)\big|
		\big(1 + |2^{m}x - \ell|\big)^{-\alpha} \dint x \\
		&\lesssim 2^{\frac{m-j}{2}}
		\big(1 + 2^{\frac{m-j}{\alpha} - 1}|k -
		2^{j-m}\ell|\big)^{-\alpha} \\
		&\lesssim 2^{\frac{j-m}{2}} |k - 2^{j-m}\ell|^{-\alpha}.
	\end{align*}

	On the other hand, if \(|2^{m}x - \ell| \leq
	2^{\frac{m-j}{\alpha}-1} |k - 2^{j-m}\ell|\),
	then we have
	\begin{align*}
		|2^{j}x -k| &= \big|2^{j-m}(2^{m}x -\ell) + 2^{j-m}\ell
		- k\big|
		\geq |2^{j-m}\ell - k| - 2^{j-m} |2^{m}x - \ell| \\
		&\geq |2^{j-m}\ell - k| - 2^{j-m} 2^{\frac{m-j}{\alpha} - 1}
		|k - 2^{j-m}\ell| \\
		&= |2^{j-m}\ell - k| \bigg(1 -
		\frac{1}{2} \cdot 2^{(m-j)(\frac{1}{\alpha} - 1)}\bigg) \\
		&\geq \frac{1}{2} \cdot |2^{j-m}\ell - k|
	\end{align*}
	because \(\alpha > 1\).
	Therefore, we can conclude that
	\begin{align*}
		|t_{\mu,\lambda}| &\lesssim 2^{\frac{j+m}{2}}
		\int_{I} \big(1 + |2^{j}x - k|\big)^{-\alpha}
		\big|\tilde{\phi}_{\mu}(x)\big| \dint x \\
		&\lesssim 2^{\frac{j-m}{2}} 
		\big(1 + |2^{j-m}\ell - k|\big)^{-\alpha} \\
		&\lesssim 2^{\frac{j-m}{2}} |k - 2^{j-m}\ell|^{-\alpha}
	\end{align*}
	also in this case.
	The estimates on \(\tilde{t}_{\mu,\lambda}\) follow in the same manner.
\end{proof}

\begin{remark}
	Lemma \ref{lm:coefficients_fafield} relies on the decay condition
	\eqref{eq:decay_behaviour_scaling} only. In particular,
	\eqref{eq:decay_behaviour_scaling} holds whenever
	\begin{equation*}
		\phi_{\lambda}(x) = 2^{\nicefrac{j}{2}} \phi(2^{j}x - k),
		\qquad |\phi(x)| \lesssim (1 + |x|)^{-\alpha},
	\end{equation*}
	which is true as soon as \(\phi\) gives rise to a regular multiresolution
	analysis in the sense of \cite{Mal89}.
	Obviously, \eqref{eq:decay_behaviour_scaling} holds also for 
	wavelets with compact support on both the primal and dual side.
	Such wavelets have been constructed on the real line and adapted 
	to the interval, cf.\ \cite{CDV93,Dau88,Dau92,Mey91} for orthonormal 
	wavelets and \cite{CDF92,DKU99} for the biorthogonal spline wavelets.
	Additionally, there are smooth wavelets on the real line which decay faster
	than any polynomial, see \cite{Mey90}.
\end{remark}

\begin{lemma}
	\label{lm:tranform_linfty}
	Suppose that \eqref{eq:decay_behaviour_scaling} holds for some \(\alpha >
	1\) on the primal and dual side.
	Let \(m \geq j_0\) and consider the
	wavelet transforms \(\Tbfm_m\) 
	and \(\tilde\Tbfm_m\) given by 
	\eqref{eq:def_of_T} and \eqref{eq:def_of_tilde_T},
	respectively. Then, for \(\nicefrac{1}{\alpha} < p \leq 2\), there holds
	\begin{equation*}
		\|\Tbfm_m^{\intercal}\|_p, 
		\|\tilde{\Tbfm}_m^{\intercal}\|_p \lesssim 1,\qquad 
		\|\Tbfm_m\|_p, 
		\|\tilde{\Tbfm}_m\|_p \lesssim 2^{m(\frac{1}{p}-\frac{1}{2})}.
	\end{equation*}
\end{lemma}

\begin{proof}
	We will show this lemma on the primal side only, as 
	the same arguments can be used for the dual side.
	First, we remark that, for \(p = 2\), 
	both statements are true since both \(\Phi\) and
	\(\Psi\) are uniformly stable bases of \(L^2([0,1])\).

	Let us consider the first inequality.
	We will start by showing this for \(p = 1\), from which the statement follows for
	all \(1 \leq p \leq 2\) by interpolation.
	In view of \eqref{eq:def_of_T}, for each \(\lambda \in \bigcup_{j = j_0}^m
	\nabla_{j}\), there holds
	\begin{equation}
		\psi_{\lambda} = \sum_{\mu \in \Delta_m} t_{\mu, \lambda}
		\phi_{\mu}.
		\label{eq:linear_combination_trafo}
	\end{equation}
	Therefore, to estimate
	\begin{equation*}
		\|\Tbfm_m^{\intercal}\|_{1} = 
		\|\Tbfm_m\|_{\infty} = \max_{\mu \in \Delta_m} \sum_{j = j_0}^m
		\sum_{\lambda \in \nabla_j}	|t_{\mu,\lambda}|,
	\end{equation*}
	we fix \(\mu = (m,\ell) \in \Delta_m\).
	In view of Lemma \ref{lm:coefficients_fafield},
	there are at most \(\Ocal(1)\) wavelets \(\psi_{\lambda}\), 
	\(\lambda = (j,k) \in \nabla_j\) with \(|k - 2^{j-m}\ell| \leq 1\),
	which implies that
	\begin{align*}
		\sum_{j=j_0}^{m} \sum_{\lambda \in \nabla_{j}}
		|t_{\mu,\lambda}|
		&\lesssim \sum_{j=j_0}^{m} 2^{\frac{j-m}{2}} 
		\left[1 + 
			\sum_{\substack{k \in \Zbbb \\ |k - 2^{j-m}\ell| \geq 1}} 
		|k - 2^{j-m}\ell|^{-\alpha}\right]
		\\[-1em]
		&\lesssim \sum_{j=j_0}^{m} 2^{\frac{j-m}{2}} \\
		&\lesssim 1
	\end{align*}
	as \(\alpha > 1\).
	Since \(\mu \in \Delta_{m}\) was arbitrary,
	this implies the first estimate for all \(1 \leq p \leq 2\) by
	interpolation.
	If \(p < 1\), we use the same arguments to conclude that
	\begin{equation*}
		\sum_{j=j_0}^{m} \sum_{\lambda \in \nabla_{j}}
		|t_{\mu,\lambda}|^{p} \lesssim \sum_{j=j_0}^{m} 2^{p
		\frac{j-m}{2}} \lesssim 1
	\end{equation*}
	as \(\alpha p > 1\),
	so Lemma \ref{lm:matrix_p_norm}
	yields the first estimate for \(\nicefrac{1}{\alpha} < p \leq 1\).

	To show the second inequality, we will first show that \(\|\Tbfm_m\|_1
	\lesssim 2^{\nicefrac{m}{2}}\). If this holds true, then we can conclude 
	for \(\frac{1}{p} = \frac{1-\theta}{1} + \frac{\theta}{2} = 1 - \frac{\theta}{2}\),
	i.e., \(1 - \theta = \frac{2}{p} -1\), by interpolation that
	\begin{equation*}
		\|\Tbfm_m\|_p \leq \|\Tbfm_m\|_1^{1-\theta} \|\Tbfm_m\|_2^{\theta}
		\lesssim 2^{\frac{m}{2}(\frac{2}{p} - 1)} \cdot 1^{\theta} =
		2^{m(\frac{1}{p} - \frac{1}{2})},
	\end{equation*}
	provided that \(1 \leq p \leq 2\).

	Let \(j_0 \leq j \leq m\) and \(\lambda  = (j,k) \in \nabla_{j}\).
	Again in view of Lemma \ref{lm:coefficients_fafield},
	there are at most \(\Ocal(2^{m-j})\) indices \(\mu = (m, \ell) \in \Delta_{m}\)
	satisfying \(|k - 2^{j-m}\ell| \leq 1\).
	Hence, there holds
	\begin{align*}
		\sum_{\mu \in \Delta_m} |t_{\mu,\lambda}|
		&\lesssim 2^{\frac{j-m}{2}} 
		\left[ 2^{m-j} + 
			\sum_{\substack{\ell \in \Zbbb \\
			|k - 2^{j-m}\ell| \geq 1}} 
		|k - 2^{j-m}\ell|^{-\alpha} \right]\\
		&= 2^{\frac{m-j}{2}} + 2^{\frac{j-m}{2}} \sum_{\substack{\ell \in \Zbbb \\
		|\ell| \geq 2^{m-j}}} 2^{\alpha(m-j)} |\ell|^{-\alpha}
		\\
		&\lesssim 2^{\frac{m-j}{2}} +
		2^{(m-j)(\alpha - \frac{1}{2})} 2^{(m-j)(1-\alpha)} \\
		&\sim 2^{\frac{m-j}{2}},
	\end{align*}
	by which \(\|\Tbfm_m\|_1 \lesssim 2^{\frac{m}{2}}\).

	Finally, for \(0 < p < 1\), we 
	use similar arguments to obtain that
	\begin{equation*}
		\sum_{\mu \in \Delta_{m}} |t_{\mu,\lambda}|^{p}
		\lesssim 2^{(m-j)(1 - \frac{p}{2})},
		\qquad \alpha p > 1,
	\end{equation*}
	which allows us to conclude the second estimate using
	Lemma \ref{lm:matrix_p_norm}.
\end{proof}

In Section \ref{sec:comparison_spaces}, 
we will need to estimate of Kronecker products of matrices.
Such estimates are derived in \cite{Hac19,LF72,Nit04}, for example. 
However, in the special case we are interested in,
the proof is rather elementary,
so we provide it for the reader's convenience.

\begin{lemma}\label{lm:kron_prod_norms}
	Let \(\Abfm \in \Rbbb^{m \times n}\) and \(\Bbfm \in \Rbbb^{k \times
	\ell}\). For \(p \in \{1, 2, \infty\}\), there holds
	\begin{equation*}
		\|\Abfm \otimes \Bbfm\|_{p} = \|\Abfm\|_{p} \|\Bbfm\|_{p}.
	\end{equation*}
\end{lemma}

\begin{proof}
	For the \(\|\cdot\|_2\)-norm, the claim follows from the fact that the
	eigenvalues of \(\Abfm \otimes \Bbfm\) are given by the products of any two
	eigenvalues of \(\Abfm\) and \(\Bbfm\).

	For the \(\|\cdot\|_{\infty}\)-norm, there holds
	\begin{align*}
		\|\Abfm \otimes \Bbfm\|_{\infty} 
		&= \max_{1 \leq r_1 \leq n} \max_{1 \leq r_2
		\leq \ell} \sum_{t_1 = 1}^{m} \sum_{t_2 =1}^{k} \big|a_{r_1, t_1}
		b_{r_2, t_2} \big|\\
		&= \left[\max_{1 \leq r_1 \leq n} \sum_{t_1 = 1}^{m} \big|a_{r_1, t_1}\big|
		\right] 
		\left[ \max_{1 \leq r_2
		\leq \ell} \sum_{t_2 =1}^{k} \big|b_{r_2, t_2} \big|\right]\\
		&= \|\Abfm\|_{\infty} \|\Bbfm\|_{\infty}.
	\end{align*}

	For the \(\|\cdot\|_{1}\)-norm, the claim follows with exactly the same
	arguments.
\end{proof}

\section{Function Spaces}
\label{sec:Function_Spaces}

The goal of this section is to characterise the function spaces used in the
setting of best \(N\)-term approximation. We will see that these spaces can be
characterised by isotropic or tensor-product wavelets, or in the case of Sobolev
spaces, by both of them. For the sake of convenience, 
for \(\lambdabfm \in \nablabfm_{\jbfm}\),
we shall denote in the following
\begin{equation*}
	|\lambdabfm|_1 \isdef |\jbfm|_1, \qquad
	|\lambdabfm|_\infty \isdef |\jbfm|_\infty.
\end{equation*}

\subsection{Sobolev Spaces}

In Section \ref{sec:Problem_Forumlation}, we have assumed 
that all \(\theta_{\mu}\) and \(\psi_{\lambdabfm}\) are scaled such 
that they are normalised in \(L^2(\square)\). 
However, if we define
\begin{equation*}
	\begin{aligned}
		\psi_{\lambdabfm}^{(p)} 
		&\isdef 2^{|\lambdabfm|_1 \big(\frac{1}{p} - \frac{1}{2} \big)} \psi_{\lambdabfm},
		&\qquad \tilde{\psi}_{\lambdabfm}^{(p)}
		&\isdef 2^{|\lambdabfm|_1 \big( \frac{1}{2} - \frac{1}{p} \big)}
		\tilde{\psi}_{\lambdabfm},\\
		\theta_{\mu}^{(p)}
		&\isdef 2^{n|\mu| \big(\frac{1}{p} - \frac{1}{2} \big)} \theta_{\mu},
		&\qquad \tilde{\theta}_{\mu}^{(p)}
		&\isdef 2^{n|\mu| \big( \frac{1}{2} - \frac{1}{p} \big)}
		\tilde{\theta}_{\mu},\\
	\end{aligned}
\end{equation*}
we get a primal basis which is normalised in \(L^p(\square)\),
and a dual basis which is normalised in \(L^{p'}(\square)\), where
\(\frac{1}{p} + \frac{1}{p'} = 1\).
With the above notation at hand, every function \(u \in L^{2}(\square)\) 
admits unique expansions
\begin{align*}
	\begin{aligned}
		u &= \sum_{\lambdabfm \in\nablabfm}
		u_{\lambdabfm}^{(p)} \psi_{\lambdabfm}^{(p)},
		&\qquad u_{\lambdabfm}^{(p)} &= \big\langle \tilde{\psi}_{\lambdabfm}^{(p)},
		\, u \big\rangle,\\
		u &= \sum_{\mu \in \hexagon} u_{\mu}^{(p)} \theta_{\mu}^{(p)}, 
		&\qquad u_{\mu}^{(p)} &= \big\langle \tilde{\theta}_{\mu}^{(p)},
		u\big\rangle.
	\end{aligned}
\end{align*}
To keep the notation 
simple, we always assume that a coefficient \(u_{\lambdabfm}\) or
\(u_{\mu}\)
without suffix \(p\) corresponds to an \(L^2(\square)\)-normalised 
expansion of \(u\). 

We intend next to characterise certain function spaces. It is 
well established, see e.g.\ \cite{Dah97,Sch98}, 
that univariate wavelet coefficients of a function 
characterise the norm of this function with respect to a range of 
function spaces. However, in a multivariate setting, we can 
also use a tensor-product wavelet basis to characterise
isotropic function spaces, i.e., there holds
\begin{align}
	\label{eq:norm_equivalence}
	\|u\|_{H^s(\square)}^2 \sim \sum_{\lambdabfm \in\nablabfm} 2^{2s |\lambdabfm|_\infty}
	\big|\langle\tilde{\psi}_{\lambdabfm},u\rangle\big|^2
	\sim \sum_{\mu \in \hexagon} 2^{2s|\mu|}\big|\langle
	\tilde{\theta}_{\mu}, u\rangle\big|^2,
	\quad - \tilde{\gamma} < s < \gamma, \\
	\label{eq:norm_equivalence_dual}
	\|u\|_{H^t(\square)}^2 \sim \sum_{\lambdabfm \in\nablabfm} 2^{2t |\lambdabfm|_\infty}
	\big|\langle u,\psi_{\lambdabfm}\rangle\big|^2
	\sim \sum_{\mu \in \hexagon} 2^{2s|\mu|} \big|\big\langle u,
	\theta_{\mu}\rangle\big|^2,
	\quad -\gamma<t<\tilde{\gamma},
\end{align}
cf.\ \cite{ACJRV15,Dah97,GOS99,HvR24,Sch98,SUV21}. 
In particular, a function \(u\) is contained in an isotropic Sovolev space
if the coefficients of its (tensor-product) wavelet 
expansion decay sufficiently fast.
Note that we have slightly abused the notation here
as we have omitted the boundary conditions which are imposed by
the primal and dual wavelet bases in \eqref{eq:norm_equivalence} and
\eqref{eq:norm_equivalence_dual}, respectively.
Note that, in accordance with \cite{DS98}, all boundary conditions which do not change
within a face of the unit cube can be expressed.

On the other hand, in \cite{GOS99}, there was also shown that
the coeffcients with respect to a tensor-product basis characterise
the dominating mixed regularity of a function \(u\). Shortly after, 
in \cite{GK00}, Sobolev spaces of \emph{hybrid regularity} were 
introduced. Roughly speaking, these spaces \(\Hfrak^{q,s}\), which are 
sometimes also called \emph{Griebel-Knapek spaces}, contain all 
functions \(u\) which, with respect to \(H^q(\square)\), admit mixed derivatives 
up to the order \(s\). In terms of tensor products, we define
these spaces in the following way.

\begin{definition}
	\label{def:Griebel_Knapek_space}
	We define \(\Hfrak^{q,s}(\square)\) for \(q \geq 0\) and \(s \in \Rbbb\)
	as the space 
	\begin{equation}
		\Hfrak^{q,s}(\square) \isdef \bigcap_{i = 1}^{n}
		\bigotimes_{j = 1}^{n} H^{s + \delta_{i,j}q}([0,1]).
		\label{eq:Griebel_Knapek_space}
	\end{equation}
	For \(q < 0\), we define the space \(\Hfrak^{q,s}(\square)\) via the duality 
	\(\Hfrak^{q,s}(\square) = \big(\Hfrak^{-q,-s}(\square)\big)'\).
\end{definition}
\begin{remark}
	In the two-dimensional setup, \eqref{eq:Griebel_Knapek_space}
	simply means that
	\begin{equation*}
		\Hfrak^{q,s}(\square) = \left(H^{q+s}([0,1]) \otimes H^{s}([0,1])\right)
		\cap \left(H^{s}([0,1]) \otimes H^{q+s}([0,1])\right).
	\end{equation*}
\end{remark}

By using the norm equivalences \eqref{eq:norm_equivalence} and
\eqref{eq:norm_equivalence_dual}, we see that we can also characterise the
spaces \(\Hfrak^{q,s}(\square)\) in terms of hyperbolic wavelet coefficients.
Note that the following norm equivalences have already been
established in \cite{GK00,GK09} for a smaller range of parameters.
Those results cover the cases \(s \geq 0\) and \(0 \leq q + s < \gamma\),
which is not sufficient to prove the Jackson and Bernstein inequalities 
in Section \ref{sec:Approximation_tensor}.

\begin{theorem}
	\label{thm:Griebel_Knapek_norm_equivalence}
	There holds
	\begin{align}
		\label{eq:norm_equivalence_GK}
		\|u\|_{\Hfrak^{q,s}(\square)}^2 &\sim \sum_{\lambdabfm \in\nablabfm}
		2^{2q|\lambdabfm|_\infty + 2s|\lambdabfm|_1}
		\big|\langle\tilde{\psi}_{\lambdabfm},u\rangle\big|^2,
		\quad -\tilde{\gamma} < s, q+s < \gamma,\\
		\label{eq:norm_equivalence_GK_dual}
		\|u\|_{\Hfrak^{q,s}(\square)}^2 &\sim \sum_{\lambdabfm \in\nablabfm}
		2^{2q|\lambdabfm|_\infty + 2s|\lambdabfm|_1}
		\big|\langle u,\psi_{\lambdabfm}\rangle\big|^2,
		\quad -\gamma < s, q+s < \tilde{\gamma}.
	\end{align}
\end{theorem}
\begin{proof}
	Let \(q \geq 0\) and \(s \in \Rbbb\). 
	For the sake of simplicity, we fix the dimension to \(n = 2\), but we
	emphasise that the same arguments also work in higher dimensions.

	We will basically follow the arguments
	that were used in \cite{HvR24} for classical Sobolev spaces. 
	By standard tensor-product
	arguments, for \(u =
	\sum_{\lambdabfm \in\nablabfm} u_{\lambdabfm} \psi_{\lambdabfm}\), there
	holds with \eqref{eq:norm_equivalence}, and \(u_{\lambdabfm} \isdef \langle
	\tilde{\psi}_{\lambdabfm}, u \rangle\),
	\begin{align*}
		\|u\|_{H^{q+s}([0,1]) \otimes H^{s}([0,1])}^2
		&\sim \sum_{\lambdabfm \in\nablabfm} 2^{2(q+s)|\lambda_1| + 2s|\lambda_2|}
		|u_{\lambdabfm}|^2, \\
		\|u\|_{H^{s}([0,1]) \otimes H^{q+s}([0,1])}^2
		&\sim \sum_{\lambdabfm \in\nablabfm} 2^{2s|\lambda_1| + 2(q+s)|\lambda_2|}
		|u_{\lambdabfm}|^2,
	\end{align*}
	provided that \(-\tilde{\gamma} < s, q+s < \gamma\). Therfore, as
	\(q+s \geq s\), there holds
	\begin{align*}
		2^{2(q+s)|\lambda_1| + 2s|\lambda_2|} + 2^{2s|\lambda_1| +
		2(q+s)|\lambda_2|}
		&\sim 2^{2(q+s)|\lambdabfm|_\infty + 2s \min\{|\lambda_1|,
		|\lambda_2|\}}\\
		&= 2^{2q|\lambdabfm|_\infty + 2s|\lambdabfm|_1}.
	\end{align*}
	Hence, \eqref{eq:norm_equivalence_GK} follows for \(q \geq 0\).
	With exactly the same arguments, we can also show
	\eqref{eq:norm_equivalence_GK_dual} for \(q \geq 0\).

	For \(q < 0\), we use a duality argument to show both asymptotic
	inequalities. By duality, there holds
	\begin{align*}
		\|u\|_{\Hfrak^{q,s}(\square)} = \sup_{\|v\|_{\Hfrak^{-q, -s}(\square)} = 1}
		\langle u, v \rangle.
	\end{align*}
	Hence, writing \(v = \sum_{\lambdabfm \in\nablabfm} \tilde{v}_{\lambdabfm}
	\tilde{\psi}_{\lambdabfm}\), we obtain by the biorthogonality that
	\begin{align*}
		\|u\|_{\Hfrak^{q,s}(\square)} &= \sup_{\|v\|_{\Hfrak^{-q, -s}(\square)} = 1}
		\sum_{\lambdabfm \in\nablabfm} u_{\lambdabfm} \tilde{v}_{\lambdabfm} \\
		&\leq \left[\sum_{\lambdabfm \in \nablabfm} 2^{2q|\lambdabfm|_\infty +
		2s |\lambdabfm|_1} |u_{\lambdabfm}|^2 \right]^{\frac{1}{2}}
		\sup_{\|v\|_{\Hfrak^{-q, -s}(\square)} = 1}
		\left[\sum_{\lambdabfm \in \nablabfm} 2^{-2q|\lambdabfm|_\infty -
		2s |\lambdabfm|_1} |\tilde{v}_{\lambdabfm}|^2 \right]^{\frac{1}{2}}\\
		&\sim 
		\left[\sum_{\lambdabfm \in \nablabfm} 2^{2q|\lambdabfm|_\infty +
		2s |\lambdabfm|_1} |u_{\lambdabfm}|^2 \right]^{\frac{1}{2}},
	\end{align*}
	where we have used \eqref{eq:norm_equivalence_GK_dual} for \(-q \geq 0\).

	For the lower bound, we remark that there holds
	\begin{align*}
		&\left[\sum_{\lambdabfm \in \nablabfm} 2^{2q|\lambdabfm|_\infty +
		2s |\lambdabfm|_1} |u_{\lambdabfm}|^2 \right]^{\frac{1}{2}}\\
		&\qquad= \sup_{\|[2^{-q|\lambdabfm|_{\infty} -s|\lambdabfm|_1}
		\tilde{v}_{\lambdabfm}]_{\lambdabfm}\|_{\ell^2(\nablabfm)} = 1} 
		\langle \ubfm, \tilde{\vbfm} \rangle_{\ell^2(\nablabfm)} \\
		&\qquad= \sup_{\|[2^{-q|\lambdabfm|_{\infty} -s|\lambdabfm|_1}
		\tilde{v}_{\lambdabfm}]_{\lambdabfm}\|_{\ell^2(\nablabfm)} = 1} 
		\left\langle u, \sum_{\lambdabfm \in\nablabfm} \tilde{v}_{\lambdabfm}
		\tilde{\psi}_{\lambdabfm} \right\rangle \\
		&\qquad\leq \|u\|_{\Hfrak^{q,s}(\square)}	
		\sup_{\|[2^{-q|\lambdabfm|_{\infty} -s|\lambdabfm|_1}
		\tilde{v}_{\lambdabfm}]_{\lambdabfm}\|_{\ell^2(\nablabfm)} = 1} 
		\left\|\sum_{\lambdabfm \in\nablabfm} \tilde{v}_{\lambdabfm}
		\tilde{\psi}_{\lambdabfm}\right\|_{\Hfrak^{-q,-s}(\square)}\\
		&\qquad\sim \|u\|_{\Hfrak^{q,s}(\square)}.
	\end{align*}
	Note that we have again used \eqref{eq:norm_equivalence_GK_dual} 
	for \(-q \geq 0\). This shows \eqref{eq:norm_equivalence_GK}.

	The same arguments also allow us to show
	\eqref{eq:norm_equivalence_GK_dual} for \(q < 0\), completing the proof of
	this theorem.
\end{proof}

\begin{remark}
	We note that the one-sided 
	upper bound in \eqref{eq:norm_equivalence_GK} can be extended for 
	\(-\tilde{d} < s, q+s < \gamma\), whereas the one-sided lower bound
	can be extended to \(-\tilde{\gamma} < s, q+s < d\). 
	Herein, \(d\) and \(\tilde{d}\) denote the order of
	polynomial exactness of \(\Psi\) and \(\tilde{\Psi}\), 
	which implies in turn that \(\Psi\) and \(\tilde{\Psi}\) admit
	\(\tilde{d}\) and \(d\) vanishing moments, respectively.
	Similarly, we can also
	extend the upper and lower bounds in \eqref{eq:norm_equivalence_GK_dual} 
	up to \(-d\) and \(\tilde{d}\), respectively.
\end{remark}

We note that there holds \(\theta_{\mu}, \psi_{\lambdabfm} \in 
	\Hfrak^{0,\gamma}(\square) \subset \Hfrak^{\gamma, 0}(\square) =
H^{\gamma}(\square)\) by the tensor-product structure.
With the above norm equivalences, it is easy to see that
\(\Hfrak^{0,\gamma}(\square)\) is continuously
embedded into \(\Hfrak^{q,s}(\square)\) if either
\(q\geq0\) and \(s<\gamma-q\),
or if \(q<0\) and \(s<\gamma\), see also \cite{BHW23}.
Furthermore, we can derive the following proposition.
\begin{proposition}
	The spaces \(\Hfrak^{q,s}(\square)\) are Hilbert spaces
	when equipped with the
	inner product
	\begin{equation*}
		\langle u,v\rangle_{\Hfrak^{q,s}(\square)} \isdef \sum_{\lambdabfm \in\nablabfm}
		2^{2q|\lambdabfm|_\infty + 2s|\lambdabfm|_1}
		\langle\tilde{\psi}_{\lambdabfm}, u\rangle \langle\tilde{\psi}_{\lambdabfm}, v\rangle.
	\end{equation*}
	Moreover, we have the Gelfand triples 
	\(\Hfrak^{q,s}(\square)\hookrightarrow
		H^q(\square)
	\hookrightarrow \Hfrak^{q,-s}(\square)\)
	for \(s \geq 0\).
\end{proposition}

\subsection{Besov Spaces}

When dealing with best \(N\)-term approximation, one 
has to consider Besov spaces, which are defined by three 
indices. Primarily, we have the regularity \(\alpha\) and the 
integrability \(p\), and secondarily, a fine index \(\tau\).
\begin{definition}
	For \(\alpha > 0\) and \(0 < p,\tau < \infty\), we define
	the quantity
	\begin{equation}
		\|\ubfm\|_{\bbfm^{\alpha, p}_{\tau}} \isdef \left[\sum_{m \geq j_0}
			2^{\tau m\alpha} \left[ \sum_{\mu \in\hexagon_m}
		|u_{\mu}|^{p}\right]^{\frac{\tau}{p}} \right]^{\frac{1}{\tau}},
		\label{eq:Besov_seminorm_isotropic}
	\end{equation}
	and \(\bbfm^{\alpha,p}_{\tau}\) as the space containing all
	vectors for which the above quantity
	\eqref{eq:Besov_seminorm_isotropic} is finite.
	For \(\max\{p,\tau\} = \infty\), we apply the usual modifications.
	Then, the \emph{isotropic Besov space} \(B^{\alpha,p}_{\tau}(\square)\) 
	is defined as the space containing all functions
	\(u = \Theta \ubfm\), for which the norm
	\begin{equation*}
		\|u\|_{B^{\alpha,p}_{\tau}(\square)} \isdef 
		\|\ubfm^{(p)}\|_{\bbfm^{\alpha,p}_{\tau}}
	\end{equation*}
	is finite.
\end{definition}
\begin{remark}
	Classically, the Besov spaces \(B^{\alpha,p}_{\tau}(\square)\) are defined
	by the decay behaviour of a function's moduli of continuity.
	However, if the wavelets involved satisfy
	\begin{enumerate}
		\item \(\theta_{\mu} \in B^{\beta,p}_{\tau}(\square)\) 
			for some \(\beta > \alpha > \max\{0, n(\nicefrac{1}{p}-1)\}\) and
		\item \(\theta_{\mu}\) has polynomial order \(d >
			\max\{\alpha, n(\nicefrac{1}{p} - 1)\}\),
	\end{enumerate}
	then the classical Besov seminorm and the discrete Besov seminorm
	\eqref{eq:Besov_seminorm_isotropic} are equivalent, cf.\
	\cite{Coh03} and the references therein.
	Besides, the Besov spaces can also be characterised by a
	Littlewood-Paley decomposition, cf.\ \cite{Tri92}.
\end{remark}
As we will see, the classical 
Besov spaces and also the Besov spaces of
dominating mixed smoothness will not be the right spaces
for the approximation with respect to a tensor-product basis
in an isotropic energy space. For this, we need to consider 
Besov spaces of hybrid regularity as introduced in \cite{BHW23}, 
which are characterised by the decay of the hyperbolic
wavelet coefficients.
\begin{definition}
	For given \(q \geq 0\), \(s > 0\), and \(0 < p,\tau < \infty\),
	we define
	\begin{equation}\label{eq:Besov_seminorm_tau}
		\|\ubfm\|_{\bfrak^{q, s, p}_\tau} \isdef 
		\left[ \sum_{\jbfm \geq\jbfm_0} 2^{\tau (q |\jbfm|_\infty + 
			s |\jbfm|_1)} \left[ \sum_{\lambdabfm \in\nablabfm_{\jbfm}}
			\big| u_{\lambdabfm} \big|^p \right]^{\frac{\tau}{p}} 
		\right]^{\frac{1}{\tau}},
	\end{equation}
	with the usual modifications in the case
	\(\max\{p,\tau\} = \infty\),
	and the space \(\bfrak^{q,s,p}_{\tau}\) as the space containing
	all vectors for which the quantity \eqref{eq:Besov_seminorm_tau} is
	finite.
	The \emph{Besov space of hybrid regularity} 
	\(\Bfrak^{q, s, p}_\tau(\square)\) is defined as
	the space containing all functions \(u = \Psi \ubfm\), for which the norm
	\begin{equation}
		\label{eq:Besov_hybrid_norm}
		\|u\|_{\Bfrak^{q,s,p}_{\tau}(\square)} \isdef
		\|\ubfm^{(p)}\|_{\bfrak^{q,s,p}_{\tau}}
	\end{equation}
	is finite.
\end{definition}
\begin{remark}
	At the first glance, the expression
	\eqref{eq:Besov_seminorm_tau} looks different
	from that in \cite{BHW23}. This can, however,
	be explained by the fact that the wavelet bases in
	\cite{BHW23} are normalised in \(L^\infty(\square)\).
	Moreover, a hyperbolic Littlewood-Paley decomposition of 
	hybrid regularity Besov spaces has recently been provided in
	\cite{ELV24}.
\end{remark}
If \(p=\tau=2\), then we have the identities
\begin{equation*}
	\Bfrak^{q,s,2}_{2}(\square) = \Hfrak^{q,s}(\square),\quad
	B^{\alpha, 2}_{2}(\square) = H^{\alpha}(\square).
\end{equation*}
Indeed, this can be immediately concluded from
\eqref{eq:Besov_seminorm_isotropic} and \eqref{eq:norm_equivalence}, or
\eqref{eq:Besov_seminorm_tau} and Theorem
\ref{thm:Griebel_Knapek_norm_equivalence}, respectively.
In particular, we have extended the norm equivalences from \cite{GK00,GK09} to
negative \(q\), and we have shown that in the case \(p=\tau=2\), the hybrid
regularity Besov spaces from \cite{BHW23} agree with the hybrid regularity Sobolev spaces.

\section{Approximation and Interpolation}
\label{sec:Approximation_Interpolation}

The aim of this section is to briefly summarise the interpolation
between vector spaces and the best
\(N\)-term approximation,
which is a kind of nonlinear approximation.
Our goal is to characterise the approximation spaces for the
anisotropic tensor-product wavelet basis.
All the results provided on interpolation and its interplay with
approximation can also be found in \cite{DeV98}.

\subsection{Interpolation}

First, we want to specify the abstract interpolation between normed vector 
spaces by means of the \(K\)-functional, cf.\ \cite{LP64}.
We consider two normed vector spaces, \((V, \| \cdot \|_V)\)
and \((W, \| \cdot \|_W)\), and we assume that
\(W\) is continuously embedded in \(V\), that is, 
\(\| \cdot \|_V \lesssim \| \cdot \|_W\).
Moreover, let \(|\cdot|_W\) be a quasi-seminorm on \(W\). 
Then, for \(u \in V\) and \(t > 0\), 
we define the \(K\)-functional by
\[
	K(u, t) \isdef \inf_{w \in W} \| u - w \|_V + t|w|_W.
\]
For \(\theta \in (0, 1)\), and \(\tau \in (0, \infty]\), we define
\[
	|u|_{(V, W)_{\theta, \tau}} 
	\isdef \left( \int_0^\infty \frac{1}{t} \big[ t^{-\theta} 
		K(u, t) \big]^\tau \dint t
	\right)^{\frac{1}{\tau}},
	\qquad \tau < \infty,
\]
and
\[
	|u|_{(V, W)_{\theta, \infty}}
	\isdef \sup_{t > 0} t^{-\theta}K(u, t).
\]
Finally, we define the interpolation spaces between \(V\) and \(W\) as
\[
	(V, W)_{\theta, \tau} \isdef \big\{ u \in V : |u|_{(V, W)_{\theta, \tau}} 
	< \infty \big\}.
\]
The term \emph{interpolation} is justified by the property
\[
	|u|_{(V, W)_{\theta, \tau}} \lesssim \| u \|_{V}^{1 - \theta}
	|u|_{W}^{\theta},
	\qquad u \in W \subset V,
\]
for example. For a proof and further results on interpolation,
consult \cite{Tri78}, for instance.

\subsection{Approximation}

A question which is closely related to interpolation is the following: 
Given a Riesz basis
\(\Psi = \{ \psi_{\lambda} : \lambda \in \nabla \}\) of a space \(V\),
for which subspace \(W \subset V\) can we approximate all functions 
\(u \in W\) with at most \(N\) terms at the rate \(N^{-s}\)? This is a topic
of \emph{nonlinear approximation} and, in particular, best \(N\)-term
approximation, as we consider trial spaces
\[
	V_N \isdef \left\{ \sum_{\lambda \in \Lambda} c_\lambda \psi_\lambda 
	: c_\lambda \in \Cbbb, \ |\Lambda| \leq N \right\}.
\]
Such trial spaces are nonlinear, as for \(u, v \in V_N\) there holds in 
general only \(u + v \in V_{2N}\), cf.\ \cite{DeV98}. 

We next define the \(N\)-term error
\[
	E_N(u) \isdef \inf_{v_N \in V_N} \| u - v_N \|_V,
\]
the seminorms
\begin{align*}
	|u|_{\Acal^s_\tau(V)} &\isdef \left( \sum_{N > 0} \frac{1}{N} \left[N^{s}
	E_N(u) \right]^{\tau} \right)^{\frac{1}{\tau}}, \quad 0 < \tau < \infty,\\
	|u|_{\Acal^s_{\infty}(V)} &\isdef \sup_{N \geq 0} N^s E_N(u),
\end{align*}
and the approximation spaces
\[
	\Acal^s_\tau(V) \isdef \left\{ u \in V : |u|_{\Acal^s_\tau(V)} < \infty \right\}.
\]

One can show that the approximation spaces \(\Acal^s_\tau(V)\) 
can, under some circumstances, be fully characterised by interpolation.
We just quote the result. For a proof, see \cite{DeV98} and 
the references therein.

\begin{theorem}
	\label{thm:approximation_interpolation}
	Consider a vector space \(V\) and a subspace \(W \subset V\), and assume 
	that there is a number \(r > 0\) such that there holds a Jackson inequality
	\[
		E_N(u) \lesssim N^{-r} |u|_W
	\]
	and a Bernstein inequality
	\[
		|u_N|_W \lesssim N^r \| u \|_V.
	\]
	Then, for each \(s \in (0, r)\) and each \(\tau \in (0, \infty]\), there holds
	\[
		\Acal^s_\tau(V) = (V, W)_{\frac{s}{r}, \tau}
	\]
	with equivalent norms.
\end{theorem}

For best \(N\)-term approximation with respect to the isotropic 
wavelet basis \(\Theta = \{\theta_{\mu} : \mu \in \hexagon\}\) 
on the unit cube, it is well known that one can derive a Jackson 
and Bernstein inequality between the spaces \(H^{q}(\square)\) and
\(B^{q+rn,\tau}_{\tau}(\square)\) where \(\frac{1}{\tau} = r +\frac{1}{2}\), 
cf.\ \cite{DeV98}. Therefore, the following theorem holds.

\begin{theorem}
	Consider the space \(H^{q}(\square)\) and the properly scaled Riesz basis
	\(\Theta = \{\theta_{\mu} : \mu \in \textnormal{\(\hexagon\)}\}\). Then, for
	\(\frac{1}{\tau} = r +\frac{1}{2}\) and \( 0 < \kappa \leq \infty\), 
	there holds
	\begin{equation*}
		A^{s}_{\kappa} \big(H^{q}(\square)\big) = \big(H^{q}(\square),\ 
		B^{q+rn,\tau}_{\tau}(\square)\big)_{\frac{s}{r},\kappa}, \quad 0 < s <
		r,
	\end{equation*}
	provided that \(\Theta \subset
	B^{q + rn + \epsilon, \tau}_{\tau}(\square)\) for some \(\epsilon > 0\)
	and that the wavelets admit polynomial exactness
	\(d > \max\{q+rn, n(r-\frac{1}{2})\}\).
\end{theorem}

\subsection{Approximation with Tensor-Product Wavelets}
\label{sec:Approximation_tensor}

With the results of the previous two subsections,
we can now classify the approximation spaces
for any function space \(H^q(\square)\) with
respect to an anisotropic tensor-product 
wavelet basis.
Note that the following proof is implicitly contained in
\cite{Han10,HS10}, but simplifies a lot in the current setting.

\begin{theorem}
	Assume that \(r > 0\), \(-\gamma < q < \gamma\), and
	\begin{align*}
		\frac{1}{\tau} = r + \frac{1}{2}.
	\end{align*}
	Then, for \(0 < s < r\) and \(0 < \kappa \leq \infty\),
	the approximation spaces with respect to tensor-product wavelets are given by
	\begin{equation}
		\Afrak^s_\kappa\big(H^q(\square) \big) 
		= \big( H^q(\square), \ \Bfrak^{q, r, \tau}_\tau(\square)
		\big)_{\frac{s}{r}, \kappa}.
		\label{eq:approximation_spaces_for_wavelets}
	\end{equation}
	\label{thm:approximation_spaces_for_wavelets}
\end{theorem}
\begin{proof}
	In accordance with Theorem \ref{thm:approximation_interpolation},
	we need to show the Jackson and Bernstein inequalities
	\begin{align}
		\label{eq:Jackson}
		\inf_{v_N \in V_N} \| u - v_N \|_{H^q(\square)}
		&\lesssim N^{-r} \|u\|_{\Bfrak^{q, r, \tau}_\tau(\square)}, &u &\in
		\Bfrak^{q,r,\tau}_{\tau}(\square),\\
		\label{eq:Bernstein}
		\|u_N\|_{\Bfrak^{q, r, \tau}_\tau(\square)} &\lesssim N^r \| u_N \|_{H^q(\square)}, 
		&u_N &\in V_N,
	\end{align}	
	to conclude the statement of the theorem.
	In order to do achieve this,
	we use an argument similar to \cite{DeV98,Nit04}.

	Let us show the Jackson inequality \eqref{eq:Jackson} first.
	We note that
	\begin{equation}
		\label{eq:Basis_Psiq}
		\Psi^q \isdef \big\{ 2^{-q |\lambdabfm|_\infty} \psi_{\lambdabfm} : 
		\lambdabfm \in\nablabfm\}
	\end{equation}
	is a Riesz basis for \(H^q(\square)\) due to the norm equivalence 
	\eqref{eq:norm_equivalence}. Thus, the best \(N\)-term approximation 
	to any function \(u\) asymptotically corresponds to the \(N\) terms 
	with the largest coefficients in absolute value. Hence, if \(\ubfm^\star 
	= (\ubfm^{\star}(k))_k\) is a descending (in absolute value) reordering 
	of the sequence \((2^{q|\lambdabfm|_\infty} u_{\lambdabfm})_{\lambdabfm}\), 
	there holds
	\begin{equation*}
		\inf_{v_N \in V_N} \| u - v_N \|_{H^q(\square)}^2 
		\sim \sum_{k = N+1}^{\infty} |\ubfm^{\star}(k)|^2.
	\end{equation*}
	Therefore, it is equivalent to require that the quantity
	\begin{equation*}
		\| \ubfm^\star \|_{\ell^\tau_{\infty}} \isdef 
		\sup_{k \in \Nbbb} k^{\frac{1}{\tau}} |\ubfm^{\star}(k)|,
		\qquad
		\frac{1}{\tau} = r + \frac{1}{2},
	\end{equation*}
	is uniformly bounded. Since the embedding \(\ell^\tau 
	\subset \ell^\tau_{\infty}\) is continuous, we have that
	\begin{align}
		\notag
		\inf_{v_N \in V_N} \| u - v_N \|_{H^q(\square)}^2 
		&\lesssim \sum_{k=N+1}^{^{\infty}} |\ubfm^{\star}(k)|^2
		\lesssim \|\ubfm^{\star}\|_{\ell^{\tau}_{\infty}(\Nbbb)}^2
		\int_{N}^{\infty} t^{-\frac{2}{\tau}} \dint t \\
		\label{eq:proof_Jackson_first}
		&\lesssim N^{-2r}\| \ubfm^{\star} \|_{\ell^\tau(\Nbbb)}^2.
	\end{align}
	By rescaling the coefficients, we obtain that
	\begin{align*}
		\| \ubfm^{\star} \|_{\ell^\tau(\Nbbb)}^\tau 
		&= \sum_{\lambdabfm \in\nablabfm} \big|2 ^{q |\lambdabfm|_\infty}
		u_{\lambdabfm}\big|^\tau
		= \sum_{\jbfm \geq \jbfm_0} 2^{\tau q |\jbfm|_\infty} \sum_{\lambdabfm \in\nablabfm_{\jbfm}}
		\big| 2^{|\jbfm|_1 \big(\frac{1}{\tau} - \frac{1}{2}\big)} u_{\lambdabfm}^{(\tau)}
		\big|^{\tau} \\
		&= \sum_{\jbfm \in \Nbbb_0^2} 2^{\tau(q |\jbfm|_\infty + r|\jbfm|_1)}
		\sum_{\lambdabfm \in\nablabfm_{\jbfm}} \big| u_{\lambdabfm}^{(\tau)} \big|^{\tau} \\
		&= \|u\|_{\Bfrak^{q, r, \tau}_{\tau}(\square)}^{\tau}.
	\end{align*}
	Together with \eqref{eq:proof_Jackson_first},
	this implies \eqref{eq:Jackson}.

	Also for the Bernstein inequality \eqref{eq:Bernstein},
	we will make use of the norm equivalence \eqref{eq:norm_equivalence}.
	Since \(u_N \in V_N\), we may write \(u_N\) as 
	a linear combination
	\begin{equation*}
		u_N = \sum_{\lambdabfm \in\Lambda} u_{\lambdabfm} \psi_{\lambdabfm}
	\end{equation*}
	of at most \(|\Lambda| \leq N\) terms.
	Again, by rescaling, we have that
	\begin{align*}
		\|u_N\|_{\Bfrak^{q, r, \tau}_\tau(\square)}^\tau
		&= \sum_{\lambdabfm \in\Lambda} 2^{\tau(q |\lambdabfm|_{\infty} + r|\lambdabfm|_1)}
		\big| u_{\lambdabfm}^{(\tau)} \big|^{\tau} 
		= \sum_{\lambdabfm \in\Lambda} 2^{\tau(q |\lambdabfm|_{\infty} + r|\lambdabfm|_1)}
		\big| 2^{|\lambdabfm|_1 \big(\frac{1}{2} - \frac{1}{p} \big)} u_{\lambdabfm} \big|^{\tau} \\
		&= \sum_{\lambdabfm \in\Lambda} 2^{\tau q |\lambdabfm|_{\infty}}
		\big| u_{\lambdabfm} \big|^{\tau}
		\lesssim \left( \sum_{\lambdabfm \in\Lambda} 1 \right)^{1 - \frac{\tau}{2}} 
		\left( \sum_{\lambdabfm \in\Lambda} 2^{2q|\jbfm|_\infty}
		\big|u_{\lambdabfm} \big|^2 \right)^{\frac{\tau}{2}} \\
		&\sim N^{r \tau} \cdot \| u \|_{H^q(\square)}^\tau.
	\end{align*}
	By taking the \(\tau\)-th root, we get \eqref{eq:Bernstein}.
\end{proof}

\section{Comparison with Isotropic Nonlinear Approximation}
\label{sec:comparison_anisotropic_isotropic}

As stated in Section \ref{sec:Approximation_Interpolation}, 
for nonlinear approximation with isotropic wavelet bases in
arbitrary dimensions, the Jackson and Bernstein inequalities 
hold for the isotropic 
Besov space \(B^{q + rn, \tau}_\tau(\square)\). Intuitively, this 
space requires more regularity on the functions, since, in 
comparison to Theorem \ref{thm:approximation_spaces_for_wavelets},
\(rn\) additional isotropic derivatives are needed instead of only \(r\) mixed
derivatives. However, the classical function space \(B^{q + rn, \tau}_{\tau}(\square)\)
cannot be characterised by tensor-product wavelets unless \(\tau = 2\), cf.\
\cite{SUV21} and the references therein. In this case, the space 
\(B^{q + rn, 2}_2(\square) = H^{q + rn}(\square)\) is a Sobolev
space and this identity can also be directly concluded from the norm
equivalences \cite{HvR24}.

To the authors' best knowledge, it is not known yet whether the 
classical Besov spaces \(B^{q+sn,p}_{\tau}(\square)\) are included 
in the Besov space of hybrid regularity \(\Bfrak^{q,s,p}_{\tau}(\square)\).
In this section, we will therefore address this question up to a certain point.

\subsection{Change of Bases}

First, we need to find a way to express the isotropic basis functions 
in terms of tensor-product functions. For simplicity, we only treat the 
case \(n = 2\) explicitly again, 
but we emphasise that the main result, 
Theorem \ref{cor:embedding_spaces}, 
carries over to the \(n\)-dimensional 
case as well if \(2s\) is replaced by \(sn\).

Therefore, we consider a set
\begin{align*}
	\Theta_{m}^{\ebfm} = \Theta_{m}^{e_1} \otimes \Theta_{m}^{e_2}.
\end{align*}
If \(e = 1\), then \(\Theta_m^{e} = \Psi_m\), so in this case, the
corresponding tensor factor is already a one-dimensional wavelet.
On the other hand, \eqref{eq:Multiscale_primal} implies that
\begin{align*}
	\big[\Theta_{m-1}^{0}, \Theta_{m}^{1}\big] =  \Theta_{m}^{0} \big[\Mbfm_{m, 0},
	\Mbfm_{m,1}\big].
\end{align*}
Hence, there holds
\begin{align*}
	\big[\Psi_{j_0}, \dots , \Psi_{m-1}\big]\otimes \Psi_{m} &=
	\Theta_{m}^{(0,1)} \big( \Tbfm_{m-1} \otimes \Ibfm\big),\\
	\Psi_m \otimes \big[\Psi_{j_0}, \dots , \Psi_{m-1}\big] &= 
	\Theta_{m}^{(1,0)} \big( \Ibfm \otimes \Tbfm_{m-1} \big),\\
	\Psi_{m}\otimes \Psi_{m} &=  \Theta_{m}^{(1,1)}\big(\Ibfm \otimes \Ibfm\big).
\end{align*}

To estimate the classical Sobolev or Besov regularity of a function \(u\) which
is discretised by tensor-product wavelets, we can therefore write
\begin{align*}
	u = \Psi \ubfm 
	&= \sum_{m = j_0}^{\infty} \left[\sum_{j = 0}^{m-1} \Psi_{(j,m)}
		\ubfm\big|_{\nablabfm_{(j,m)}} + \Psi_{(m,j)}
	\ubfm\big|_{\nablabfm_{(m,j)}} \right] + \Psi_{(m,m)} 
	\ubfm\big|_{\nablabfm_{(m,m)}}\\
	&= \sum_{m = j_0}^{\infty} \Theta_{m}^{(0,1)} \left( \Tbfm_{m-1} \otimes 
	\Ibfm \right) \ubfm\big|_{\bigcup_{j < m} \nabla_j \times \nabla_m}\\
	&\qquad\qquad + \Theta_{m}^{(1,0)} \left(\Ibfm \otimes \Tbfm_{m-1} \right) 
	\ubfm\big|_{\bigcup_{j < m} \nabla_m \times \nabla_j}
	+\Theta_{m}^{(1,1)} 
	\ubfm\big|_{\nabla_m \times \nabla_m}.
\end{align*}
On the other hand, we can also express the tensor-product coefficients in terms
of the isotropic ones, meaning that
\begin{align*}
	u = \Theta \vbfm &= \sum_{m = j_0}^{\infty} 
	\Theta_m^{(0,1)} \vbfm\big|_{\hexagon_m^{(0,1)}} +
	\Theta_m^{(1,0)} \vbfm\big|_{\hexagon_m^{(1,0)}} +
	\Theta_m^{(1,1)} \vbfm\big|_{\hexagon_m^{(1,1)}}\\
	&= \sum_{m = j_0}^{\infty} 
	\big( [\Psi_{j_0} \dots \Psi_{m-1}] \otimes \Psi_m\big)
	(\tilde{\Tbfm}_{m-1}^{\intercal} \otimes \Ibfm)
	\vbfm\big|_{\hexagon_m^{(0,1)}}\\
	&\qquad + \big( \Psi_m \otimes [\Psi_{j_0} \dots \Psi_{m-1}]\big)
	(\Ibfm \otimes \tilde{\Tbfm}_{m-1}^{\intercal})
	\vbfm\big|_{\hexagon_m^{(1,0)}}
	+ \big( \Psi_m \otimes \Psi_m\big) \vbfm\big|_{\hexagon_m^{(1,1)}}.
\end{align*}

Hence, if \(\frac{1}{\tau} = s + \frac{1}{2}\),
any function \(u = \Psi \ubfm = \Theta \vbfm\) satisfies
\begin{align}
	\notag
	\|u\|_{\Bfrak^{q,s,\tau}_{\tau}(\square)}^{\tau}
	&= \|\ubfm^{(\tau)}\|_{\bfrak^{q,s,\tau}_{\tau}}^{\tau}
	= \sum_{\jbfm \geq \jbfm_0} 2^{\tau q|\jbfm|_{\infty} + \tau s|\jbfm|_1}
	\|\ubfm^{(\tau)}\|_{\ell^{\tau}(\nablabfm_{\jbfm})}^{\tau}\\
	\notag
	&= \sum_{m = j_0}^{\infty} \sum_{j = j_0}^{m-1}
	2^{\tau qm + \tau (s+\frac{1}{2} -\frac{1}{\tau})(j+m)}
	\left(\|\ubfm\|_{\ell^{\tau}(\nablabfm_{(j,m)})}^{\tau}
	+\|\ubfm\|_{\ell^{\tau}(\nablabfm_{(j,k)})}^{\tau} \right)\\
	\notag
	&\qquad + 2^{\tau m q + 2\tau m(s +\frac{1}{2} - \frac{1}{\tau})}
	\|\ubfm\|_{\ell^{\tau}(\nablabfm_{(m,m)})}^{\tau}\\
	\notag
	&= \sum_{m = j_0}^{\infty} 
	2^{\tau qm}
	\left(
		\big\|(\tilde{\Tbfm}_{m-1}^{\intercal} \otimes \Ibfm)
		\vbfm|_{\hexagon_{m}^{(0,1)}}\big\|_{\ell^{\tau}(\bigcup_{j < m}
		\nabla_j \times \nabla_m)}^{\tau}
	\right. \\
	&\qquad + \left.
		\big\|(\Ibfm \otimes \tilde{\Tbfm}_{m-1}^{\intercal})
		\vbfm|_{\hexagon_{m}^{(1,0)}}\big\|_{\ell^{\tau}(\bigcup_{j < m}
		\nabla_m \times \nabla_j)}^{\tau}
	\right) + 
	2^{\tau mq}
	\|\vbfm\|_{\ell^{\tau}(\hexagon_{m}^{(1,1)})}^{\tau},
	\label{eq:Besov_norm_coordinate_transform}
\end{align}
and conversely, also
\begin{align}
	\notag
	\|u\|_{B^{\alpha,\tau}_{\tau}(\square)}^{\tau}
	&= \|\vbfm^{(\tau)}\|_{\bbfm^{\alpha,\tau}_{\tau}}^{\tau}
	= \sum_{m = j_0}^{\infty} 2^{\tau\alpha m}
	\|\vbfm^{(\tau)}\|_{\ell^{\tau}(\hexagon_m)}^{\tau}\\
	\notag
	&= \sum_{m=j_0}^{\infty} 2^{\tau m (\alpha + 2(\frac{1}{2} -
	\frac{1}{\tau}))} \left( \|\vbfm\|_{\ell^{\tau}(\hexagon_m^{(0,1)})}
		+\|\vbfm\|_{\ell^{\tau}(\hexagon_m^{(1,0)})}
	+\|\vbfm\|_{\ell^{\tau}(\hexagon_m^{(1,1)})}\right)\\
	\notag
	&= \sum_{m = j_0}^{\infty} 2^{\tau m (\alpha - 2s)}
	\left( \big\|(\Tbfm_{m-1} \otimes \Ibfm) \ubfm\big|_{\bigcup_{j<m}
		\nabla_j \times \nabla_m}
	\big\|_{\ell^{\tau}(\hexagon_m^{(0,1)})}^{\tau} \right.\\
	&\qquad + \left. \big\|(\Ibfm \otimes \Tbfm_{m-1})  \ubfm\big|_{\bigcup_{j<m}
		\nabla_m \times \nabla_j}
		\big\|_{\ell^{\tau}(\hexagon_m^{(1,0)})}^{\tau} 
	+ \|\ubfm\|_{\ell^{\tau}(\nablabfm_{(m,m)})}^{\tau}\right).
	\label{eq:Besov_norm_inverse_transform}
\end{align}

\subsection{Comparison of Approximation Spaces}
\label{sec:comparison_spaces}

The goal of this section is to compare the classical approximation spaces
\(B^{q+sn,\tau}_{\tau}(\square)\) with the approximation spaces for
tensor-product wavelets \(\Bfrak^{q,s,\tau}_{\tau}(\square)\). We will see that
there is a range of regularity spaces whose elements can be approximated by
tensor-product wavelets but not by isotropic wavelets.

If a bit of additional smoothness is available, 
with the help of \cite{ACJRV15}, we can immediately conclude that
the approximation with tensor-product wavelets performs
at least as well.
\begin{proposition}
	\label{prop:embeddings_epsilon}
	For any \(\epsilon > 0\), \(s \geq 0\), and \(u = \Psi \ubfm\), there holds
	with \(\frac{1}{\tau} = s + \frac{1}{2}\),
	\begin{equation*}
		\|u\|_{B^{q + s - \epsilon, \tau}_{\tau}(\square)}
		\lesssim \|u\|_{\Bfrak^{q,s,\tau}_{\tau}(\square)}
		\lesssim \|u\|_{B^{q + 2s + \epsilon, \tau}_{\tau}(\square)}.
	\end{equation*}
\end{proposition}

\begin{remark}
	Proposition \ref{prop:embeddings_epsilon} can be improved by replacing the
	\(\epsilon\) by a logarithmic correction weight.
	In order to do so, Besov spaces with logarithmic correction need to be
	introduced, cf.\ \cite{ACJRV15}.
	Moreover, note that the results \cite{ACJRV15} in a two-dimensional
	setting.
\end{remark}

Proposition \ref{prop:embeddings_epsilon} immediately implies that 
there are a lot of functions which satisfy
\(u \in \Afrak^s\big(H^{q}(\square)\big)\)
but \(u \notin A^s\big(H^{q}(\square)\big)\), 
where \(\Afrak^s\) and \(A^s\) denote 
the approximation spaces with respect to the tensor-product and 
the isotropic wavelet bases, respectively.
In particular, all functions which admit slightly more regularity than
minimally required to be in \(A^s\big(H^{q}(\square)\big)\) are also in
\(\Afrak^s\big(H^{q}(\square)\big)\).
In the next theorem, we will however show that no additional
regularity is required.
\begin{theorem}
	\label{cor:embedding_spaces}
	Suppose that the primal and dual scaling functions and 
	wavelets satisfy the asymptotic decay estimate
	\eqref{eq:decay_behaviour_scaling}
	for some \(\alpha > 1\).
	Moreover, let \(s < \alpha - \frac{1}{2}\) and
	\(q+s, q+sn > \max\{0,\, n(s - \frac{1}{2})\}\). 
	Then, for \(\frac{1}{\tau} = s +\frac{1}{2}\), 
	the embeddings 
	\begin{equation*}
		B^{q+sn,\tau}_{\tau}(\square) \hookrightarrow
		\Bfrak^{q,s,\tau}_{\tau}(\square)\hookrightarrow
		B^{q+s,\tau}_{\tau}(\square)
	\end{equation*}
	are continuous.
\end{theorem}
\begin{proof}
	Since \(q+s, q+sn > \max\{0,\, n(\frac{1}{\tau}-1)\}\),
	both \(B^{q+s,\tau}_{\tau}(\square)\) and
	\(B^{q+sn,\tau}_{\tau}(\square)\)
	can be characterised by a sufficiently regular wavelet basis,
	cf.\ \cite{Coh03}.
	Therefore, it suffices to show that, for \(u = \Psi \ubfm = \Theta \vbfm\),
	\begin{equation*}
		\|\vbfm^{(\tau)}\|_{\bbfm^{q+s,\tau}_{\tau}}
		\lesssim 
		\|\ubfm^{(\tau)}\|_{\bfrak^{q,s,\tau}_{\tau}}
		\lesssim
		\|\vbfm^{(\tau)}\|_{\bbfm^{q+sn,\tau}_{\tau}}.
	\end{equation*}

	Let us first set \(n = 2\) and consider \(u \in
	B^{q+2s,\tau}_{\tau}(\square)\).
	We want to
	apply a suitable coordinate transforms to derive the estimate
	\begin{align*}
		\|\ubfm^{(\tau)}\|_{\bfrak^{q,s,\tau}_{\tau}} \lesssim
		\|\vbfm^{(\tau)}\|_{\bbfm^{q+2s,\tau}_{\tau}} = \left[ \sum_{m = 0}^{\infty}
			2^{\tau m (q + 2s)} \sum_{\mu \in \hexagon_m} 
			\big|v_{\mu}^{(\tau)}\big|^{\tau}
		\right]^{\frac{1}{\tau}}.
	\end{align*}
	In view of \eqref{eq:Besov_norm_coordinate_transform}, 
	there holds
	\begin{align*}
		\|\ubfm^{(\tau)}\|_{\bfrak^{q,s,\tau}_{\tau}}^{\tau} 
		&= \sum_{m = j_0}^{\infty} 
		2^{\tau m q}
		\left(\big\|(\tilde{\Tbfm}_{m-1}^{\intercal} \otimes \Ibfm)
			\vbfm|_{\hexagon_m^{(0,1)}}\big\|_{\ell^{\tau}(
			\bigcup_{j < m} \nabla_j \times \nabla_m)}^{\tau}
		\right. \\
		&\qquad + \left.
			\big\|(\Ibfm \otimes \tilde{\Tbfm}_{m-1}^{\intercal})
			\vbfm|_{\hexagon_m^{(1,0)}}\big\|_{\ell^{\tau}(\bigcup_{j < m} \nabla_m
			\times \nabla_j)}^{\tau}
		\right) + 
		2^{\tau mq}\|\vbfm\|_{\ell^{\tau}(\hexagon_{m}^{(1,1)})}^{\tau}.
	\end{align*}

	If \(\tau < 1\),
	we have due to Lemma
	\ref{lm:matrix_p_norm} and Lemma \ref{lm:kron_prod_norms}
	\begin{align*}
		\big\|\tilde{\Tbfm}_{m-1}^{\intercal} \otimes \Ibfm
		\big\|_{\tau}^{\tau}
		\leq
		\big\|\big(\tilde{\Tbfm}_{m-1}^{\odot
		\tau} \otimes \Ibfm \big)^{\intercal} \big\|_{1}
		\leq \big\|\big(\tilde{\Tbfm}_{m-1}^{\odot
		\tau}\big)^{\intercal} \big\|_{1},
	\end{align*}
	which can be shown to be uniformly bounded using the arguments of the proof
	of Lemma \ref{lm:tranform_linfty}, 
	as \(s < \alpha - \frac{1}{2}\) ensures that 
	\(\tau > \frac{1}{\alpha}\).
	If \(1 \leq \tau \leq 2\), and \(\frac{1}{\tau} = \frac{1-\theta}{1} +
	\frac{\theta}{2}\), interpolation and Lemma \ref{lm:kron_prod_norms} yield
	\begin{align*}
		\big\|\tilde{\Tbfm}_{m-1}^{\intercal} \otimes \Ibfm
		\big\|_{\tau}^{\tau}
		&\leq 
		\big\|\tilde{\Tbfm}_{m-1}^{\intercal} \otimes \Ibfm
		\big\|_{1}^{(1-\theta)\tau} \cdot
		\big\|\tilde{\Tbfm}_{m-1}^{\intercal} \otimes \Ibfm
		\big\|_{2}^{\theta\tau}\\
		&\leq
		\big\|\tilde{\Tbfm}_{m-1}^{\intercal}
		\big\|_{1}^{(1-\theta)\tau} \cdot
		\big\|\tilde{\Tbfm}_{m-1}^{\intercal}
		\big\|_{2}^{\theta\tau},
	\end{align*}
	which is uniformly bounded by Lemma \ref{lm:tranform_linfty}.
	After applying the same estimate to \(\Ibfm \otimes
	\tilde{\Tbfm}_{m-1}^{\intercal}\),
	we can conclude that
	\begin{align*}
		\|\ubfm^{(\tau)}\|_{\bfrak^{q,s,\tau}_{\tau}}^{\tau}
		&\lesssim \sum_{m=j_0}^{\infty}
		2^{\tau m q} \|\vbfm\|_{\ell^{\tau}(\hexagon_m)}^{\tau}
		= \sum_{m = j_0}^{\infty} 2^{\tau m (q + 2s)}
		\|\vbfm^{(\tau)}\|_{\ell^{\tau}(\hexagon_m)}^{\tau},
	\end{align*}
	which is what we wanted to show.

	For the other estimate, we use \eqref{eq:Besov_norm_inverse_transform}
	to see that
	\begin{align}
		\notag
		\|\vbfm^{(\tau)}\|_{\bbfm^{q+s,\tau}_{\tau}}
		&= \sum_{m = j_0}^{\infty} 2^{\tau m (q-s)}
		\left( \big\|(\Tbfm_{m-1} \otimes \Ibfm) \ubfm\big|_{\bigcup_{j<m}
			\nabla_j \times \nabla_m}
		\big\|_{\ell^{\tau}(\hexagon_m^{(0,1)})}^{\tau} \right.\\
		&\qquad + \left. \big\|(\Ibfm \otimes \Tbfm_{m-1})  \ubfm\big|_{\bigcup_{j<m}
			\nabla_m \times \nabla_j}
			\big\|_{\ell^{\tau}(\hexagon_m^{(1,0)})}^{\tau} 
		+ \|\ubfm\|_{\ell^{\tau}(\nablabfm_{(m,m)})}^{\tau}\right).
		\label{eq:inverse_embedding_first}
	\end{align}
	Also here, by using the arguments of the Lemmata \ref{lm:matrix_p_norm},
	\ref{lm:kron_prod_norms}, and \ref{lm:tranform_linfty}, we see that
	\begin{equation*}
		\big\|\Tbfm_{m-1} \otimes \Ibfm \big\|_{\tau}^{\tau}
		\leq \big\|\Tbfm_{m-1}^{\odot \tau} \otimes \Ibfm
		\big\|_{1} \leq \big\|\Tbfm_{m-1}^{\odot \tau} \big\|_{1}
		\lesssim 2^{\tau m (\frac{1}{\tau} - \frac{1}{2})}
		= 2^{\tau m s},
	\end{equation*}
	because \(\frac{1}{\alpha} <\tau \leq 1\).
	If \(1 \leq \tau \leq 2\) and \(\frac{1}{\tau} = \frac{1-\theta}{1} +
	\frac{\theta}{2}\), meaning that \(1-\theta = 2s\), 
	we may interpolate again to deduce that
	\begin{align*}
		\big\|\Tbfm_{m-1} \otimes \Ibfm \big\|_{\tau}^{\tau}
		\leq 
		\big\|\Tbfm_{m-1} \otimes \Ibfm \big\|_{1}^{2s\tau} \cdot 
		\big\|\Tbfm_{m-1} \otimes \Ibfm \big\|_{2}^{\theta\tau}
		\lesssim 
		\big\|\Tbfm_{m-1}\big\|_{1}^{2s\tau} 
		\lesssim 2^{\tau m s}.
	\end{align*}
	Since this applies for \(\Ibfm \otimes \Tbfm_{m-1}\) in the same style,
	there finally holds
	\begin{align*}
		\|\vbfm^{(\tau)}\|_{\bbfm^{q+s,\tau}_{\tau}} \lesssim \sum_{m=j_0}^{\infty}
		2^{\tau m q} \sum_{|\jbfm|_{\infty} = m}
		\|\ubfm\|_{\ell^{\tau}(\nablabfm_{\jbfm})}^{\tau}
		= \sum_{\jbfm \geq \jbfm_0} 2^{\tau q |\jbfm|_\infty + \tau s
		|\jbfm|_1}
		\|\ubfm^{(\tau)}\|_{\ell^{\tau}(\nablabfm_{\jbfm})}^{\tau}.
	\end{align*}

	In the case \(n \geq 2\), we apply Lemma \ref{lm:kron_prod_norms}
	recursively to matrices arising from an \(n\)-fold tensor products.
	Since \(\|(\tilde{\Tbfm}_{m}^{\odot \tau})^{\intercal}\|_{1}\)
	is uniformly bounded, the right estimate follows.

	For the left estimate, we need to bound tensor products of at most
	\((n-1)\) tensor factors \(\Tbfm_{m-1}\) and at least one identity,
	resulting in a norm asymptotically bounded by \(2^{\tau m (n-1)s}\).
	On the other hand, the weight in \eqref{eq:inverse_embedding_first} results
	from the rescaling of the coefficients from
	\(L^{\tau}(\square)\) to \(L^{2}(\square)\), 
	which gives rise to a factor \(2^{\tau m (q - (n-1)s)}\)
	in \(n\) dimensions.
	Hence, this multiplies to \(2^{\tau m q}\), as in the two-dimensional case.
\end{proof}
\begin{remark}
	The maximal approximation rate \(s\) is limited by the decay behaviour of
	the primal and dual scaling functions.
	However, if compactly supported wavelets are used, \(\alpha\), and therefore
	also \(s\), can be chosen arbitrarily large, by which this limitation can be
	mitigated.
	The same holds true for wavelets that decay faster than any polynomial.
\end{remark}

\section{Conclusion}
\label{sec:conclusion}

We have extended the wavelet characterisations of the hybrid regularity Sobolev
spaces \(\Hfrak^{q,s}(\square)\). Therefrom outgoing, we have shown that the
approximation spaces \(\Afrak^s(H^{q}(\square))\), with respect to
tensor-product wavelets, correspond to sequences in
\(\bfrak^{q,s,\tau}_{\tau}\) with \(\frac{1}{\tau} = s +\frac{1}{2}\). 
These sequence spaces
characterise the seminorms of the Besov spaces of
hybrid regularity \(\Bfrak^{q,s,\tau}_{\tau}(\square)\).
Finally, we have shown by elementary coordinate transforms that all functions
in \(B^{q+sn,\tau}_{\tau}(\square) \subset A^s(H^{q}(\square))\) can also be
approximated at least at the same rate 
\(N^{-s}\) by \(N\)-term tensor-product wavelets.
Although this seems natural, this was not known up to now. 
Moreover, for positive regularity, we have shown the embedding
\(B^{q+sn,\tau}_{\tau}(\square) \hookrightarrow
\Bfrak^{q,s,\tau}_{\tau}(\square)\), meaning that an isotropic space is included
in a space of dominating mixed smoothness.

On the other hand, also the other natural embedding
\(\Bfrak^{q,s,\tau}_{\tau} \hookrightarrow B^{q+s,\tau}_{\tau}(\square)\)
was shown to be continuous.
In all these proofs, merely estimates on the coordinate transforms between
tensor-product wavelets and isotropic wavelets have been used.
To the authors' best knowledge, this technique has, up to now, not been applied
to investigate Besov spaces of hybrid regularity. 
Therefore, it might provide new insight into a
whole range of function spaces. 

\subsection*{Acknowledgement}
The authors thank B\'{e}atrice Vedel (Universit\'{e} de Bretagne Sud) 
and Tino Ullrich (Technische Universit\"{a}t Chemnitz) 
for fruitful discussions and
putting the articles \cite{ACJRV15,SUV21} into the authors' attention.

\bibliographystyle{plain}
\bibliography{literature}

\begin{thebibliography}{10}

\bibitem{ACJRV15}
Patrice Abry, Marianne Clausel, St\'{e}phane Jaffard, St\'{e}phane~G.\ Roux,
  and B\'{e}atrice Vedel.
\newblock The hyperbolic wavelet transform: an efficient tool for multifractal
  analysis of anisotropic fields.
\newblock {\em Revista Matemática Iberoamericana}, 31(1):313--348, 2015.

\bibitem{BG04}
Hans-Joachim Bungartz and Michael Griebel.
\newblock Sparse grids.
\newblock {\em Acta Numerica}, 13:147--269, 2004.

\bibitem{BHW23}
Glenn Byrenheid, Janina H\"{u}bner, and Markus Weimar.
\newblock Rate-optimal sparse approximation of compact break-of-scale
  embeddings.
\newblock {\em Applied and Computational Harmonic Analysis}, 65:40--66, 2023.

\bibitem{Coh03}
Albert Cohen.
\newblock {\em Numerical Analysis of Wavelet Methods}, volume~32 of {\em
  Studies in Mathematics and Its Applications}.
\newblock Elsevier Science B.V., Amsterdam, 1st edition, 2003.

\bibitem{CDF92}
Albert Cohen, Ingrid Daubechies, and Jean-Christophe Feauveau.
\newblock Biorthogonal bases of compactly supported wavelets.
\newblock {\em Communications on Pure and Applied Mathematics}, 45(5):485--560,
  1992.

\bibitem{CDV93}
Albert Cohen, Ingrid Daubechies, and Pierre Vial.
\newblock Wavelets on the interval and fast wavelet transforms.
\newblock {\em Applied and Computational Harmonic Analysis}, 1(1):54--81, 1993.

\bibitem{Dah97}
Wolfgang Dahmen.
\newblock Wavelet and multiscale methods for operator equations.
\newblock {\em Acta Numerica}, 6:55–228, 1997.

\bibitem{DKU99}
Wolfgang Dahmen, Angela Kunoth, and Karsten Urban.
\newblock Biorthogonal spline wavelets on the interval -- stability and moment
  conditions.
\newblock {\em Applied and Computational Harmonic Analysis}, 6(2):132--196,
  1999.

\bibitem{DS98}
Wolfgang Dahmen and Reinhold Schneider.
\newblock Wavelets with complementary boundary conditions -- function spaces on
  the cube.
\newblock {\em Results in Mathematics}, 34(3–4):255--293, 1998.

\bibitem{Dau88}
Ingrid Daubechies.
\newblock Orthonormal bases of compactly supported wavelets.
\newblock {\em Communications on Pure and Applied Mathematics}, 41(7):909--996,
  1988.

\bibitem{Dau92}
Ingrid Daubechies.
\newblock {\em Ten Lectures on Wavelets}.
\newblock CBMS-NSF Regional Conference Series in Applied Mathematics. Society
  for Industrial and Applied Mathematics, Philadelphia, Pennsylvania, 1992.

\bibitem{DeV98}
Ronald~A. DeVore.
\newblock Nonlinear approximation.
\newblock {\em Acta Numerica}, 7:51–150, 1998.

\bibitem{DL93}
Ronald~A. DeVore and George~G. Lorentz.
\newblock {\em Constructive Approximation}, volume 303 of {\em Grundlehren der
  mathematischen Wissenschaften}.
\newblock Springer, Berlin-Heidelberg, 1st edition, 1993.

\bibitem{ELV24}
C\'{e}line Esser, Laurent Loosveldt, and B\'{e}atrice Vedel.
\newblock Regularity of weighted tensorized fractional {B}rownian fields and
  associated function spaces.
\newblock Preprint, arXiv:2412.03366, 2024.

\bibitem{GH13}
Michael Griebel and Helmut Harbrecht.
\newblock On the construction of sparse tensor product spaces.
\newblock {\em Mathematics of Computation}, 82(282):975--994, 2013.

\bibitem{GK00}
Michael Griebel and Stephan Knapek.
\newblock Optimized tensor-product approximation spaces.
\newblock {\em Constructive Approximation}, 16:525--540, 2000.

\bibitem{GK09}
Michael Griebel and Stephan Knapek.
\newblock Optimized general sparse grid approximation spaces for operator
  equations.
\newblock {\em Mathematics of Computation}, 78(268):2223--2257, 2009.

\bibitem{GOS99}
Michael Griebel, Peter Oswald, and Thomas Schiekofer.
\newblock Sparse grids for boundary integral equations.
\newblock {\em Numerische Mathematik}, 83(2):279--312, 1999.

\bibitem{Hac19}
Wolfgang Hackbusch.
\newblock {\em Tensor Spaces and Numerical Tensor Calculus}, volume~56 of {\em
  Springer Series in Computational Mathematics}.
\newblock Springer, Cham, 2nd edition, 2019.

\bibitem{Han10}
Markus Hansen.
\newblock {\em Nonlinear Approximation and Function Spaces of Dominating Mixed
  Smoothness}.
\newblock PhD thesis, Friedrich-Schiller-Universit{\"a}t Jena, 2010.

\bibitem{HS10}
Markus Hansen and Winfried Sickel.
\newblock Best \(m\)-term approximation and tensor products of sobolev and
  besov spaces -- the case of non-compact embeddings.
\newblock Preprint 39, DFG-SPP 1324, 2010.

\bibitem{HvR24}
Helmut Harbrecht and Remo {von Rickenbach}.
\newblock Compression of boundary integral operators discretized by anisotropic
  wavelet bases.
\newblock {\em Numerische Mathematik}, 156(3):853--899, 2024.

\bibitem{LF72}
Peter Lancaster and Hanafi~K. Farahat.
\newblock Norms on direct sums and tensor products.
\newblock {\em Mathematics of Computation}, 26(118):401--414, 1972.

\bibitem{LP64}
Jacques-Louis Lions and Jaak Peetre.
\newblock Sur une classe d'espaces d'interpolation.
\newblock {\em Publications math\'{e}matiques de l'I.H.\'{E}.S.}, 19:5--68,
  1964.

\bibitem{Mal89}
St\'{e}phane~G. Mallat.
\newblock Multiresolution approximations and wavelet orthonormal bases of
  \({L}^2(\mathbb{R})\).
\newblock {\em Transactions of the American Mathematical Society},
  315(1):69--87, 1989.

\bibitem{Mey90}
Yves Meyer.
\newblock {\em Ondelettes et Op\'{e}rateurs I: Ondelettes}.
\newblock Actualit\'{e}s Math\'{e}matiques. Hermann, Paris, 1990.

\bibitem{Mey91}
Yves Meyer.
\newblock Ondelettes sur l’intervalle.
\newblock {\em Revista Matem\'{a}tica Iberoamericana}, 7(2):115--133, 1991.

\bibitem{NS17}
Van~Kien Nguyen and Winfried Sickel.
\newblock Isotropic and dominating mixed {B}esov spaces: {A} comparison.
\newblock In Michael Cwikel and Mario Milman, editors, {\em Functional
  Analysis, Harmonic Analysis, and Image Processing: A Collection of Papers in
  Honor of Björn Jawerth}, volume 693 of {\em Contemporary Mathematics}, pages
  363--389. American Mathematical Society, Philadelphia, 2017.

\bibitem{Nit04}
P\'{a}l-Andrej Nitsche.
\newblock {\em Sparse Tensor Product Approximation of Elliptic Problems}.
\newblock PhD thesis, ETH Z{\"u}rich, 2004.
\newblock ETH Diss Nr.~15795.

\bibitem{Nit06}
P\'{a}l-Andrej Nitsche.
\newblock Best \({N}\) term approximation spaces for tensor product wavelet
  bases.
\newblock {\em Constructive Approximation}, 24(1):49--70, 2006.

\bibitem{SS11}
Stefan~A.\ Sauter and Christoph Schwab.
\newblock {\em Boundary Element Methods}, volume~39 of {\em Springer Series in
  Computational Mathematics}.
\newblock Springer, Berlin-Heidelberg, 2011.

\bibitem{Sch98}
Reinhold Schneider.
\newblock {\em Multiskalen- und Wavelet-Matrixkompression: analysisbasierte
  Methoden zur effizienten L\"osung gro{\ss}er vollbesetzter
  Gleichungssysteme}.
\newblock B.G.~Teubner, Stuttgart, 1st edition, 1998.

\bibitem{SUV21}
Martin Schäfer, Tino Ullrich, and B\'{e}atrice Vedel.
\newblock Hyperbolic wavelet analysis of classical isotropic and anisotropic
  {B}esov–{S}obolev spaces.
\newblock {\em Journal of Fourier Analysis and Applications}, 27(3):51, 2021.

\bibitem{SU09}
Winfried Sickel and Tino Ullrich.
\newblock Tensor products of {S}obolev–{B}esov spaces and applications to
  approximation from the hyperbolic cross.
\newblock {\em Journal of Approximation Theory}, 161(2):748--786, 2009.

\bibitem{Ste08}
Olaf Steinbach.
\newblock {\em Numerical Approximation Methods for Elliptic Boundary Value
  Problems}.
\newblock Springer, New York, 2008.

\bibitem{Tri78}
Hans Triebel.
\newblock {\em Interpolation Theory, Function Spaces, Differential Operators},
  volume~18 of {\em North-Holland Mathematical Library}.
\newblock North-Holland Publishing Company, Amsterdam New York Oxford, 1978.

\bibitem{Tri92}
Hans Triebel.
\newblock {\em Theory of Function Spaces II}.
\newblock Modern Birk\"{a}user Classics. Birkh\"{a}user, Basel, 1st edition,
  1992.

\bibitem{Vyb05}
Jan Vyb\'{i}ral.
\newblock {\em Function spaces with dominating mixed smoothness}.
\newblock PhD thesis, Friedrich-Schiller-Universität Jena, 2005.

\end{thebibliography}

\end{document}